\documentclass[11pt]{article}

\usepackage[utf8]{inputenc}
\usepackage{lmodern}
\usepackage[normalem]{ulem} 
\usepackage{float}
\usepackage[margin=0.8in]{geometry}
\usepackage{enumitem}
\usepackage{array}

\usepackage[utf8]{inputenc} 
\usepackage{tocloft}
\usepackage{url}
\usepackage{graphicx, titlesec}
\usepackage{amscd,amsmath,amsthm,amssymb}
\usepackage{verbatim}
\usepackage[all]{xy}
\usepackage{hyperref}
\usepackage{booktabs}
\usepackage{tikz-cd}
\usetikzlibrary{patterns}
\usepackage{ytableau}
\usepackage[norelsize,lined,ruled, linesnumbered ]{algorithm2e} 
\usepackage{booktabs}
\usepackage{cleveref}
\usepackage{natbib}

\newcommand{\RR}{\mathbb{R}}
\newcommand{\CC}{\mathbb{C}}
\newcommand{\PP}{\mathbb{P}}
\newcommand{\Vol}{\mathrm{Vol}}
\definecolor{ao}{rgb}{0.0, 0.5, 0.0}
\definecolor{myred}{rgb}{0.81, 0.09, 0.13}

\newtheorem{theorem}{Theorem}[section]

\theoremstyle{definition}
\newtheorem{definition}[theorem]{Definition}

\newtheorem{example}[theorem]{Example}
\newtheorem{remark}[theorem]{Remark}

\theoremstyle{plain}
\newtheorem{lemma}[theorem]{Lemma}
\newtheorem{proposition}[theorem]{Proposition}
\newtheorem{corollary}[theorem]{Corollary}

\title{Vandermonde Cells Through the Lens of Positive Geometry}
\author{Fatemeh Mohammadi and Sebastian Seemann}
\date{}

\begin{document}

\maketitle

\begingroup
\renewcommand\thefootnote{}
\footnotetext{%
\noindent\textbf{2020 MSC:} 52B11, 14M25, 05E40, 14N10, 52B20.\\
\textbf{Keywords:} Vandermonde cells, positive geometries, canonical forms, boundary stratification, polyhedral subdivisions.}
\endgroup

\begin{abstract}
We study the geometric and algebraic structure of Vandermonde cells, defined as images of the standard probability simplex under the Vandermonde map given by consecutive power sum polynomials. Motivated by their combinatorial equivalence to cyclic polytopes, which are well-known examples of positive geometries and tree amplituhedra, we investigate whether Vandermonde cells admit the structure of positive geometries. 
We derive explicit parametrizations and algebraic equations for their boundary components, extending known results from the planar case to arbitrary dimensions. By introducing a mild generalization of the notion of positive geometry, allowing singularities within boundary interiors, we show that planar Vandermonde cells naturally fit into this extended framework.
Furthermore, we study Vandermonde cells in the setting of Brown-Dupont's mixed Hodge theory formulation of positive geometries, and show that they form a genus zero pair.

These results provide a new algebraic and geometric understanding of Vandermonde cells, establishing them as promising examples within the emerging theory of positive geometries.

\end{abstract}

{\hypersetup{linkcolor=black}
\setcounter{tocdepth}{1}
\setlength\cftbeforesecskip{1.1pt}
{\tableofcontents}}
\section{Introduction}

For integers $n,d\geq 1$, the Vandermonde cell $\Pi_{n,d}$ is defined as the image of the standard probability simplex
$\Delta_{n-1} := \{x \in \mathbb{R}^n_{\ge 0} : \|x\|_1 = 1\}$
under the Vandermonde map
\[
\nu_{n,d} : x \longmapsto (p_1(x), p_2(x), \ldots, p_d(x)),
\quad \text{where } p_k(x) = \sum_{i=1}^n x_i^k
\]
are the power-sum polynomials. 
This map has been studied extensively in connection with hyperbolic polynomials~\cite{arnol1986hyperbolic, givental1987moments, kostov1989geometric, kostov1999hyperbolicity, meguerditchian1992theorem}, as well as in relation to its restriction to the nonnegative orthant~\cite{1959_Ursell}. Both the fibers and images of $\nu_{n,d}$ have been analyzed using computational algebraic geometry~\cite{bik2021semi, Recoveryfrompowersums}. 
It was recently shown~\cite{WonderfulGeometry} that Vandermonde cells share the combinatorial type of cyclic polytopes.

Cyclic polytopes are simple examples of tree amplituhedra and positive geometries~\cite{arkani2017positive}, mathematical structures that have recently emerged as central objects in the study of scattering amplitudes in quantum field theory and cosmology. A positive geometry is a pair $(X, X_{\geq 0})$ consisting of a complex algebraic variety $X$ defined over $\mathbb{R}$ and a real semi-algebraic region $X_{\geq 0} \subset X_\mathbb{R}$, equipped with a distinguished meromorphic differential form, the \emph{canonical form}, whose logarithmic poles lie precisely along the boundary of $X_{\geq 0}$. The canonical form encodes the physical and combinatorial data associated with the geometry and its boundaries.

Examples of positive geometries include projective polytopes and polypols~\cite{kohn2025adjoints}. Beyond these, a wide range of objects such as toric varieties, cluster varieties, and Grassmannians can also be equipped with the structure of a positive geometry. More generally, pairs of varieties together with logarithmic volume forms have been investigated from a birational perspective in~\cite{BurnsideRings}.

\medskip
\noindent{\bf Motivation.}
The structural parallels between cyclic polytopes and Vandermonde cells naturally suggest that the latter may also possess a positive-geometry structure. This work investigates under which conditions Vandermonde cells define positive geometries in the sense of~\cite{arkani2017positive, brown2025positive}. 
In the planar case, we identify explicit subdivisions and determine their canonical forms, showing that they fit within a mild generalization of positive geometries that allows singularities within boundary interiors. 
In higher dimensions, we study their structure in the framework of~\cite{brown2025positive}, showing that they define a genus zero pair.

\medskip
\noindent{\bf Main results.}
Our main results analyze the geometric and algebraic structure of Vandermonde cells and connect them to the emerging theory of positive geometries.

\medskip
\noindent
\textbf{(1) Boundary parametrization and defining equations.}
Using a theorem of Ursell~\cite{1959_Ursell} (Theorem~\ref{UrsellTheorem}), we obtain an explicit parametric description of all boundary components of Vandermonde cells (Lemma~\ref{lem:Parametrisation for all d}). 
We further derive defining algebraic equations for these components, providing a complete characterization of their algebraic boundaries. 
In the planar case, these computations yield closed-form expressions (Section~\ref{Equations planar boundaries}), while in higher dimensions we implement symbolic computations in \textsc{Macaulay2} to study the structure and singularities of the boundaries (Section~\ref{Computations in Maccauley}).

\medskip
\noindent
\textbf{(2) Planar Vandermonde cells as positive geometries.}
By slightly relaxing the definition of positive geometries to allow singularities in the interiors of boundary components, we show that planar Vandermonde cells fall within this broader framework. We construct explicit canonical forms by subdividing planar Vandermonde cells into elementary regions, which provide fundamental examples of generalized positive geometries (Subsection~\ref{subsec:Planar Vandermonde cells}). 
Related extensions have been independently studied in~\cite{brown2025positive}.

\medskip
\noindent
\noindent
\textbf{(3) Canonical forms of planar Vandermonde cells.}
In Section~\ref{sec:positive_geometry}, we develop the explicit construction of canonical forms for planar Vandermonde cells within an \emph{extended definition} of positive geometries that allows controlled singularities inside boundary components.
We first identify the singularities and subdivision patterns of planar cells and then express their canonical forms as rational differential forms obtained by summing contributions from elementary subregions.
This extension preserves the essential axioms of positive geometries, particularly linearity and residue compatibility, while accommodating algebraic regions beyond the convex setting.

\medskip\noindent
\textbf{(4) Hodge-theoretic interpretation.}
We reinterpret Vandermonde cells through the mixed Hodge-theoretic formulation of positive geometries developed by Brown and Dúpont~\cite{brown2025positive}. 
Using the unirationality of their boundary hypersurfaces, we prove that all boundaries of Vandermonde cells have vanishing genus (Theorem~\ref{Thm:Vanishing geometric genus}), and consequently that the pair $(\PP^{d-1}, \Pi_{n,d})$ is a genus-zero pair (Corollary~\ref{cor:genus-zero}). 
This ensures the existence of canonical forms in the Brown–Dúpont sense and situates Vandermonde cells within a broader algebro-geometric framework of positive geometries (Section~\ref{Section:BrownDupont-setup}).

\medskip
\noindent
\textbf{(5) The limiting Vandermonde cell.}
We analyze the limiting object ${\rm lim}_{n\rightarrow \infty}\Pi_{n,d}$, obtained as the increasing limit of Vandermonde cells for fixed dimension~$d-1$. 
In the planar case, the boundary converges to the cuspidal cubic $\{y^2 = x^3\}$, while accumulating infinitely many singular points, which implies that the topological boundary of the limit $\Pi_d$ is not semi-algebraic. 
Consequently, no rational differential form can have poles exactly along $\partial\Pi_d$, so the limiting cell does not admit a classical rational canonical form. 
Instead, Section~\ref{sec:limitingcell} develops an alternative description based on \emph{dual volumes}: using polytope approximations, we identify classes of compact subsets of~$\mathbb{R}^d$ whose canonical forms can be expressed through the dual-volume representation of the canonical form of a polytope. 
This approach provides a geometric framework for extending canonical forms beyond the convex or rational setting.

\medskip
Taken together, these results establish Vandermonde cells as natural algebraic models within the theory of positive geometries. 
They provide a unified perspective linking explicit geometric parametrizations, canonical forms, and Hodge-theoretic structures.

\medskip
\noindent{\bf Outline of the paper.}
Section~\ref{sec:vandermonde-cells} reviews the definition and main properties of Vandermonde cells, following~\cite{WonderfulGeometry}. We extend the known boundary parametrizations from the planar case to arbitrary dimensions and enumerate the hypersurfaces contributing to the boundary. We also show that the irreducible boundary components of Vandermonde cells are unirational varieties, a fact that will play a central role in the study of the pair $(\PP^n, \mathrm{bd}\,\Pi_{n,d})$ and its genus. The section concludes with explicit equations for planar boundaries and practical instructions for computing non-planar boundaries using computer algebra software.
Section~\ref{sec:positive_geometry} reviews the theory of positive geometries, including a generalized framework that allows singularities within boundary interiors, as developed in~\cite{brown2025positive}. We present relevant examples and discuss canonical form computations, emphasizing the residue method in singular settings~\cite{weber2005residue}. We then apply this to planar Vandermonde cells, computing their canonical forms via subdivisions and analyzing their boundary structures.
Section~\ref{Section:BrownDupont-setup} introduces the Brown–Dúpont framework of mixed Hodge structures and establishes that Vandermonde cells define genus-zero pairs in this setting, thereby ensuring the existence of canonical forms and formal properties such as functoriality and multiplicativity.
Finally, Section~\ref{sec:limitingcell} studies the limiting Vandermonde cell ${\rm lim}_{n\to\infty}\Pi_{n,d}$ and its non-semi-algebraic boundary, proposing an alternative formulation of canonical forms based on dual volumes.

\medskip\noindent{\bf Acknowledgement.}
The authors are grateful to Sebastian Debus for insightful and helpful discussions. F.M. was partially supported by FWO grants (G023721N, G0F5921N), and the grant iBOF/23/064 from KU Leuven. S.S. is supported by the FWO PhD fellowship (11PEP24N), and the FWO grants G0F5921N (Odysseus) and G023721N.  

\section{Algebraic boundaries of Vandermonde cells}\label{sec:vandermonde-cells}

In this section we develop a detailed description of the boundary geometry of Vandermonde cells.  
We first establish explicit parametrizations of the boundary components and analyze their algebraic and geometric properties (Theorem~\ref{UrsellTheorem} and Lemma~\ref{lem:Parametrisation for all d}).  
We then prove that all boundary hypersurfaces are unirational (Proposition~\ref{prop:Unirational}) and illustrate how their defining equations can be computed symbolically using \texttt{Macaulay2} (Example~\ref{Computations in Maccauley}).  
These are the preliminary results needed for what follows.

\subsection{Definitions and general boundary description}

We begin by fixing notation and recalling basic definitions used throughout the paper.  
For a function $f$, we write $\{f=0\}$ for the set of all points $x$ satisfying $f(x)=0$.  
By a \emph{cusp}, we mean a singular point of a planar cubic curve that is not an ordinary node.  
Equivalently, a plane cubic curve $C$ has a cusp if its Weierstrass form $y^2=p(x)$ is defined by a polynomial $p(x)$ with a triple root; we refer to such curves as \emph{cuspidal}.  
By the \emph{topological boundary} of a semi-algebraic set in a projective variety, we mean its boundary with respect to the Euclidean topology.

\medskip

For $n,k\in\mathbb{N}$ with $n\geq k$, let
\[
e_k := \sum_{I\in I_k}\prod_{i\in I}x_i,
\qquad I_k := \{\, I\subseteq [n] : |I|=k \,\}
\]
denote the $k$-th elementary symmetric polynomial in $n$ variables, and let 
$p_a := \sum_{i=1}^n x_i^a$ be the $a$-th power-sum polynomial.
The \emph{Vandermonde map} in dimension~$d$ on~$n$ variables is defined by
\[
\nu_{n,d}\colon \mathbb{R}^n \longrightarrow \mathbb{R}^d,
\qquad
x \longmapsto (p_1(x),p_2(x),\ldots,p_d(x)).
\]
The corresponding \emph{$(n,d)$-Vandermonde cell} is
$\Pi_{n,d} := \nu_{n,d}(\Delta_{n-1})$,
where $\Delta_{n-1} := \{\,x\in\mathbb{R}^n_{\ge0} : x_1+\cdots+x_n=1\,\}$ is the standard probability simplex.  
Since the first coordinate of every image point equals~$1$, we regard $\Pi_{n,d}$ as a subset of~$\mathbb{R}^{d-1}$ via projection onto the last $d-1$ coordinates.  
By the Tarski–Seidenberg theorem, Vandermonde cells $\Pi_{n,d}$ are semi-algebraic sets.

\medskip

To describe these images more concretely, it is convenient to restrict to the nonnegative Weyl chamber
\[
W_n := \{\, x\in\mathbb{R}^n : 0\le x_1\le x_2\le\cdots\le x_n\,\},
\qquad
W_{n-1} := W_n\cap\Delta_{n-1}.
\]
By symmetry, one has $\Pi_{n,d} = \nu_{n,d}(W_{n-1})$.  
Fibers of this map, and of its various extensions, have been studied in~\cite{Recoveryfrompowersums, RegularsequencesofSymmetricpolynomials}.  
Distinct rational points lying in the same fiber correspond to solutions of the classical Prouhet–Tarry–Escott problem, linking the arithmetic structure of Vandermonde cells to deep Diophantine phenomena.

\medskip

The following classical result characterizes the boundary points of a Vandermonde cell; 
see~\cite{1959_Ursell} and~\cite[Theorem~2.4]{WonderfulGeometry}.

\begin{theorem}[\cite{1959_Ursell, WonderfulGeometry}]\label{UrsellTheorem}
The boundary $\,\mathrm{bd}\,\Pi_{n,d}$ of the Vandermonde cell is the image under the Vandermonde map $\nu_{n,d}$ of points of one of the following two types:
\begin{enumerate}
    \item 
      $(\underbrace{0,\dots,0}_{m_0},
      \underbrace{x_1}_{m_1},
      \underbrace{x_2,\dots,x_2}_{m_2},
      \dots,
      \underbrace{x_{d-1},\dots,x_{d-1}}_{m_{d-1}})$,
      where $m_{2k-1}=1$, $m_0\ge0$, and $m_{2k}\ge1$;
    \item 
      $(\underbrace{x_1,\dots,x_1}_{m_1},
      \underbrace{x_2}_{m_2},
      \dots,
      \underbrace{x_{d-1},\dots,x_{d-1}}_{m_{d-1}})$,
      where $m_{2k}=1$ and $m_{2k-1}\ge1$.
\end{enumerate}
Here each $m_i$ denotes the multiplicity of the corresponding value~$x_i$, that is, the number of consecutive repetitions of~$x_i$.  
Moreover, the preimages are unique up to permutation of coordinates.
\end{theorem}

The above theorem is used in~\cite[Theorem~2.21]{WonderfulGeometry} to parametrize the boundaries of Vandermonde cells in the planar case.  
Since this result serves as one of our main computational tools, we briefly summarize the key ideas of its proof, following~\cite[Section~2]{WonderfulGeometry}, to keep the exposition self-contained.  
The argument proceeds by analyzing the critical loci of the last power-sum polynomial $p_{d-1}^\alpha$ restricted to the Vandermonde variety 
\[
V_k^{d}(c) := \nu_{n,d}^{-1}(c)\cap\mathbb{R}_{\ge0}^n, \qquad c \in \{1\}\times\mathbb{R}_{\ge 0}^k.
\]
Using Descartes' rule of signs, one shows that points with more than $d-1$ distinct nonzero coordinates are smooth.  
The critical points of $p_{d-1}^\alpha$ are then identified as those with exactly $d-1$ distinct coordinates, via a Lagrange-multiplier argument.  
By Morse-theoretic considerations, these critical configurations correspond precisely to the two multiplicity patterns in Theorem~\ref{UrsellTheorem}, which in turn parametrize the boundary components of the Vandermonde cell~$\Pi_{n,d}$.

\medskip

Theorem~\ref{UrsellTheorem} leads to explicit parametrizations for the boundary components of Vandermonde cells. 

\begin{lemma}\label{lem:Parametrisation for all d}
The boundary components of the Vandermonde cells $\Pi_{n,d}$ are parametrized by
\[
\left(
  \sum_{i=1}^{d-2} m_i x_i^j
  \;+\;
  \left( 
    \frac{1}{m_{d-1}} 
    - \sum_{i=1}^{d-2} \frac{m_i}{m_{d-1}} x_i
  \right)^{\! j}
\right)_{j=2,\dots,d},
\]
under the following conditions:
\begin{itemize}
    \item either $m_0 \ge 0$, $m_{2k-1}=1$, $m_{2k}\ge1$, \quad or \quad $m_0=0$, $m_{2k}=1$, $m_{2k-1}\ge1$;
    \item the variables $(x_1,\dots,x_{d-1})$ satisfy $0 \le x_i \le x_{i+1}$;
    \item the normalization condition $\sum_{i=1}^{d-1} m_i x_i = 1$ holds.
\end{itemize}
\end{lemma}

\begin{proof}
By Theorem~\ref{UrsellTheorem}, preimages of boundary components are of two types.  
The first type has the form
\[
x = (\underbrace{0,\dots,0}_{m_0}, 
      \underbrace{x_1}_{m_1}, 
      \underbrace{x_2,\dots,x_2}_{m_2}, 
      \dots, 
      \underbrace{x_{d-1},\dots,x_{d-1}}_{m_{d-1}}),
\]
while the second type is
\[
x = (\underbrace{x_1,\dots,x_1}_{m_1}, 
      \underbrace{x_2}_{m_2}, 
      \dots, 
      \underbrace{x_{d-1},\dots,x_{d-1}}_{m_{d-1}}).
\]
Evaluating the Vandermonde map on either type of preimage gives
$\textstyle{\nu_{n,d}(x)
  = \left( \sum_{i=1}^{d-1} m_i x_i^j \right)_{j=2,\dots,d}}$.
Using the constraint $\sum_{i=1}^{d-1} m_i x_i = 1$, we can eliminate $x_{d-1}$ as
\[
x_{d-1}
  = \frac{1}{m_{d-1}}
    - \sum_{i=1}^{d-2} \frac{m_i}{m_{d-1}}\,x_i,
\]
which yields the stated parametrization.

\medskip
Each of the three conditions in the statement of Lemma~\ref{lem:Parametrisation for all d} reflects a distinct geometric property:
\begin{itemize}
    \item The first condition ensures that the image of the parametrization corresponds precisely to boundary points, as characterized in Theorem~\ref{UrsellTheorem};
    \item The second guarantees injectivity of the parametrization when restricted to the Weyl chamber;
    \item The third encodes the fact that the domain of the Vandermonde map $\nu_{n,d}$ is the simplex~$\Delta_{n-1}$.
\end{itemize}
\end{proof}

\begin{remark}
Without loss of generality, one may assume that all multiplicities~$m_i$ with indices of fixed parity (depending on the parity of~$d$) are equal to~$1$.  
In the proof of Lemma~\ref{lem:Parametrisation for all d}, one can then solve for a variable~$x_i$ with multiplicity~$1$ to obtain a polynomial parametrization of the boundary components.
\end{remark}

With these general boundary parametrizations in hand, we turn to the global 
structure of $\mathrm{bd}\,\Pi_{n,d}$.

\subsection{Combinatorial and geometric properties}
We next study global geometric and combinatorial features of the boundaries of Vandermonde cells.

\medskip
The following result from~\cite[Theorem~3.1]{WonderfulGeometry} shows that each Vandermonde cell~$\Pi_{n,d}$ shares the \emph{combinatorial type} of the cyclic polytope~$C(n,d-1)$, defined as the convex hull of $n$ points on the $(d-1)$-dimensional real moment curve
\[
t \longmapsto (t, t^2, \dots, t^{d-1}).
\]
In the planar case, the topological boundary components of~$\Pi_{n,3}$ are concave, and hence the convex hull ${\rm conv}(\Pi_{n,3})$ forms a cyclic polytope.  
This convex hull can be interpreted as the \emph{amplituhedron} for $k=1$ with
\[
Z = \begin{pmatrix}
1 & \tfrac{1}{k} & \tfrac{1}{k^2}
\end{pmatrix}_{k=1,\dots,n},
\]
as introduced in~\cite{arkani2017positive}.  
Consequently, the limiting Vandermonde cell may be regarded as a specific \emph{limiting amplituhedron}, distinct from the asymptotic construction studied in~\cite{koefler2025taking}.

\begin{theorem}[\cite{WonderfulGeometry}, Theorem~3.1]\label{Thm:TypeofCyclicPolytope}
There exists a homeomorphism
\[
\Phi : \mathrm{bd}\,C(n,d-1) \;\longrightarrow\; \mathrm{bd}\,\Pi_{n,d},
\]
which restricts to a diffeomorphism on the interior of every face of~$C(n,d-1)$.
\end{theorem}

\medskip

In particular, by McMullen's Upper Bound Theorem, the first half of the $f$-vector of~$\Pi_{n,d}$ coincides with that of the cyclic polytope:
\[
f_i(C(n,d-1)) = \binom{n}{i},
\]
which is maximal among all $(d-1)$-dimensional polytopes with a fixed number of vertices.


\begin{remark}
    From the viewpoint of \emph{positive geometry}, cyclic polytopes are key examples of tree-level amplituhedra; see~\cite{arkani2017positive}.  
More general amplituhedra are closely related to cyclic polytopes and their stabbing chambers~\cite{seemann2025stab}; see also~\cite{lam2014totally} for an introduction to amplituhedra and their central role in the theory of positive geometries.  
Determining the canonical forms associated with these geometries remains a central open problem, with one of the key challenges being the analysis of the boundary hypersurfaces involved.
\end{remark}

Different boundary components of a Vandermonde cell~$\Pi_{n,d}$ may lie on the same algebraic hypersurface.  
In Lemma~\ref{NumberOfBoundaries}, we characterize precisely when this occurs and determine the number of distinct algebraic hypersurfaces that constitute the boundary of~$\Pi_{n,d}$.  
As an immediate consequence, we observe that the Zariski closures of these boundary hypersurfaces are, in general, \emph{non-normal} varieties.

\medskip

For positive integers $a$ and $b$, $p(a,b)$ denotes the number of partitions of~$a$ into~$b$ positive parts.

\begin{lemma}\label{NumberOfBoundaries}
The boundary of the Vandermonde cell $\Pi_{n,d}$ is defined by 
\[
p\!\left(
  n - \Big\lfloor \tfrac{d-1}{2} \Big\rfloor,\;
  \Big\lceil \tfrac{d-1}{2} \Big\rceil
\right)
\]
distinct equations, together with all boundary equations of $\Pi_{n-1,d}$.  
More concretely, for a given multiplicity vector $(m_0,\ldots,m_{d-1})$, all vectors
\[
\{(m_{\sigma(0)},\ldots,m_{\sigma(d-1)})\}_{\sigma\in S_{d-1}}
\]
satisfying the constraints of Theorem~\ref{UrsellTheorem} correspond to boundary components lying on the same algebraic hypersurface.  
In general, two multiplicity vectors $(m_0,\ldots,m_{d-1})$ and $(m'_0,\ldots,m'_{d-1})$ determine boundary components on the same hypersurface if and only if they represent the same partition of~$n$.
\end{lemma}

\begin{proof}
By Theorem~\ref{UrsellTheorem}, every boundary point of $\Pi_{n,d}$ arises as the image under the Vandermonde map $\nu_{n,d}$ of a point of type~(1) or~(2).  
Since $\nu_{n,d}$ is symmetric, any $x=(x_1,\dots,x_n)\in\mathbb{R}^n$ has the same image as its permutation 
$x_{\sigma}=(x_{\sigma(1)},\dots,x_{\sigma(n)})$ for $\sigma\in S_n$.  
Hence, all polynomial relations satisfied by $\nu_{n,d}(x)$ are invariant under such permutations.

Moreover, the particular domain point representing a given multiplicity pattern does not affect the algebraic relations among its image coordinates.  
Therefore, two multiplicity vectors $(m_1,\dots,m_{d-1})$ and $(m'_1,\dots,m'_{d-1})$ define the same boundary hypersurface if and only if they encode the same partition of~$n$.  
This observation also implies that all boundary equations arising from type~(1) preimages correspond to boundaries of Vandermonde cells with smaller~$n$.

To determine the new boundary equations appearing as~$n$ increases, it suffices to count the distinct partitions arising from multiplicity vectors of type~(2).  
Such a vector $(m_1,\dots,m_{d-1})$ satisfies $m_k=1$ for every even~$k$, and therefore represents a partition $(m_1,\dots,m_{d-1})\vdash n$ containing exactly 
$\lfloor \tfrac{d-1}{2} \rfloor$ parts equal to~$1$.  
By removing the entries $m_i$ with even indices, we obtain a partition of
\[
n - \Big\lfloor \tfrac{d-1}{2} \Big\rfloor
\quad \text{into} \quad
\Big\lceil \tfrac{d-1}{2} \Big\rceil
\text{ parts.}
\]

Finally, since by Theorem~\ref{UrsellTheorem} the preimages of boundary points are unique up to permutation, none of the new boundary hypersurfaces obtained for larger~$n$ coincides with those from smaller~$n$.  
Hence the number of new boundary equations is precisely the number of such partitions, as claimed.
\end{proof}

We illustrate this in the case $d=4$, where the boundary hypersurfaces arise from type~(2) preimages.

\begin{example}
For $d=4$, the type~(2) preimages have the form
\[
(\underbrace{x_1,\dots,x_1}_{m_1},\,x_2,\,\underbrace{x_3,\dots,x_3}_{n-m_1-1})\quad \text{with}\quad 
1 \le m_1 \le n-2,
\quad
0 < x_1 < x_2 < x_3 < 1,
\]
and
$m_1 x_1 + x_2 + (n - m_1 - 1)x_3 = 1$.
Solving for $x_2$ gives
$x_2 = 1 - m_1 x_1 - (n - m_1 - 1)x_3$, 
which provides a parametrization of the image $\nu_{n,4}(x)$ as
\[
\bigl(
m_1 x_1^2 + (1 - m_1 x_1 - (n - m_1 - 1)x_3)^2 + (n - m_1 - 1)x_3^2,\;
m_1 x_1^3 + (1 - m_1 x_1 - (n - m_1 - 1)x_3)^3 + (n - m_1 - 1)x_3^3
\bigr).
\]
Gr\"obner basis computations based on this parametrization yield the explicit boundary equations.  
By Lemma~\ref{NumberOfBoundaries}, the hypersurface obtained for~$m_1$ coincides with that obtained for~$n - m_1 - 1$.

\medskip
Similarly, type~(1) preimages are of the form
$$(\underbrace{0,\dots,0}_{k},\,x_1,\,\underbrace{x_2,\dots,x_2}_{n-k-2},\,x_3)\quad \text{with}\quad 0<x_1<\tfrac{n-k}{n}, \quad x_1<x_2<x_3\quad \text{and}\quad x_3>\tfrac{n-k}{n}.$$
Their images under $\nu_{n,4}$ likewise parametrize boundary components of~$\Pi_{n,4}$.
\end{example}

\medskip

The boundary parametrizations derived above naturally raise a structural question about the algebraic nature of these hypersurfaces.  
In particular, we now show that the boundary components are not only algebraically defined but also exhibit strong birational properties.  
Recall that a projective variety~$X$ is  \emph{unirational} if there exists a dominant rational map 
$\Phi\colon \PP^n \dashrightarrow X$.  
We next show that all boundary hypersurfaces of Vandermonde cells are unirational, a fact that will play an important role in Section~\ref{Section:BrownDupont-setup}.

\begin{proposition}\label{prop:Unirational}
All boundary hypersurfaces of the Vandermonde cell~$\Pi_{n,d}$ are unirational.  
In particular, boundary curves and surfaces are rational.
\end{proposition}

\begin{proof}
Let $m=(m_1,\dots,m_{d-1})$ be a multiplicity vector corresponding to a boundary component~$D$, and let 
$\Delta_{n-1}^m$ denote the subset of the simplex consisting of points with the given multiplicity vector.  
Consider the restriction of the Vandermonde map to this subset, 
\[
\nu_{\vert_{\Delta_{n-1}^m}}\colon \Delta_{n-1}^m \longrightarrow \mathbb{R}^{d-1},
\]
and extend it to the $(d-2)$-dimensional linear subspace
\[
V = \mathcal{V}\!\left(
  \{\,x_1=\dots=x_{m_1},\,
    \dots,\,
    x_{(\sum_{i=1}^{d-2}m_i)+1}=\dots=x_n,\,
    \textstyle\sum_{i=1}^{n}x_i=1
  \,\}
\right)
\subset \mathbb{R}^n,
\]
whose points share the same multiplicity pattern.
Denote by $\bar{\nu}_m$ the extension of $\nu_{\vert_{\Delta_{n-1}^m}}$ to complex coefficients, viewed as a rational map between projective spaces:
\[
\bar{\nu}_m\colon \PP^{d-2} \dashrightarrow \PP^{d-1}.
\]
The image of $\bar{\nu}_m$ is contained in an irreducible projective variety, and since the image of a polynomial map from a projective variety is again projective, the Zariski closure of this image is a hypersurface.  
Hence $\bar{\nu}_m$ is a dominant rational map.

By~\cite[Proposition~2.9]{Conca2009}, the sequence of $d-1$ consecutive weighted power-sum polynomials in $d-1$ variables forms a regular sequence.  
The argument there extends directly to the positively weighted case, implying that the components of the polynomial map~$\bar{\nu}_m$ also form a regular sequence.  
Such polynomial maps have also been studied in~\cite{Recoveryfrompowersums}.
A regular sequence of $d-1$ polynomials cuts out a complete intersection of codimension~$d-1$ in~$\PP^{d-1}$, whose dimension equals the transcendence degree of its function field.  
Consequently, no polynomial relation exists among the $d-1$ consecutive weighted power-sum polynomials, confirming that $\bar{\nu}_m$ is dominant.

Since $\bar{\nu}_m$ is a rational map between projective varieties of the same dimension, it follows that $\bar{\nu}_m$ is generically finite.  
Therefore, the image hypersurface corresponding to~$D$ is unirational.
(Determining which unirational varieties are actually rational is the classical \emph{Lüroth problem}, which is completely resolved in dimensions one and two: all unirational curves and surfaces are rational.)
\end{proof}

Having established the global structure and birational properties of the boundaries, we now pass to explicit equations and computations.

\subsection{Explicit equations and computational examples}

In the context of positive geometry, boundary equations are necessary in order to determine canonical forms explicitly. 
We next outline how boundary equations can be computed symbolically using the computer algebra system \texttt{Macaulay2}.  
The general procedure is to evaluate the Vandermonde map at preimages of a fixed multiplicity type, eliminate the auxiliary variables, and obtain the defining equation of the corresponding boundary hypersurface.  
A complete implementation of these computations, including routines for arbitrary~$n$ and~$d$, is available in the public GitHub repository:  
\url{https://github.com/SebSeemann/Boundaries-of-Vandermonde-cells}

\begin{example}\label{Computations in Maccauley}

For $d=4$, the Vandermonde map is
\[
\nu_{4,4}(x_1,x_2,x_3,x_4)
 = (p_2(x),p_3(x),p_4(x))
 = (x_1^2+\cdots+x_4^2,\;
    x_1^3+\cdots+x_4^3,\;
    x_1^4+\cdots+x_4^4),
\]
with image coordinates $y_1=p_2(x)$, $y_2=p_3(x)$, and $y_3=p_4(x)$.  
We compute the boundary equations of the first non-planar Vandermonde cell $\Pi_{4,4}$ using \texttt{Macaulay2}.

In this case, there are two valid partitions corresponding to boundary hypersurfaces.  
Computing the boundary for the partition associated with the type-(2) multiplicity vector $(2,1,1)$, we obtain a principal ideal in $y_1,y_2,y_3$.  
The generator of this ideal is the polynomial $P(y_1,y_2,y_3)$:
\begin{align*}
P(y_1,y_2,y_3) ={}&\,y_1^6 - 684y_1^5 + 1536y_1^4y_2 - 544y_1^3y_2^2 - 720y_1^4y_3 + 1209y_1^4 - 4168y_1^3y_2 \\
& + 4224y_1^2y_2^2 - 576y_1y_2^3 - 192y_2^4 + 3096y_1^3y_3 - 6336y_1^2y_2y_3 + 1152y_1y_2^2y_3  \\
&+ 2304y_1^2y_3^2 - 796y_1^3+ 3144y_1^2y_2 - 4512y_1y_2^2 + 2496y_2^3 - 2700y_1^2y_3  \\
&+ 8208y_1y_2y_3 - 7200y_2^2y_3 - 3744y_1y_3^2 + 6912y_2y_3^2 - 2304y_3^3 + 210y_1^2   \\
&- 648y_1y_2+ 544y_2^2 + 576y_1y_3 - 1008y_2y_3 + 468y_3^2 - 24y_1 + 40y_2 - 36y_3 + 1.
\end{align*}
For $n=3$ and the multiplicity vector $(1,1,1)$, one obtains the simpler relation
\[
Q(y_1,y_2,y_3) = 3y_1^2 - 6y_1 + 8y_2 - 6y_3 + 1 = 0.
\]
By Theorem~\ref{UrsellTheorem}, the boundary of $\Pi_{4,4}$ is cut out by the equations $P(y_1,y_2,y_3)=0$ and $Q(y_1,y_2,y_3)=0$, where the hypersurface $\{P=0\}$ consists of three boundary components of dimension~$2$, corresponding to the three multiplicity vectors defining the same partition.
\end{example}

Analogous computations can be carried out for $n=5$.  
There are two partitions of $4$ of size~$2$, namely $(3,1)\vdash 4$ and $(2,2)\vdash 4$, and the corresponding boundary equations can be obtained directly using the code provided in the GitHub repository. 

\begin{remark}
Elimination and substitution do not, in general, commute.  
Elimination is performed via the computation of a Gr\"obner basis followed by projection, and the resulting semi-algebraic set need not be Zariski closed, and so taking its closure may introduce extraneous components.  

For instance, substituting $n=4$ first and then eliminating $x_1,x_2$ yields a polynomial of degree~$8$ in $y_1,y_2,y_3$, whereas substituting $m=1$ afterwards leaves the degree unchanged.  
By contrast, substituting $n=4$ and $m=1$ before elimination produces a degree~$6$ polynomial, which explicit computations confirm as the correct boundary equation.  
Similarly, substituting $m=1$ first and then eliminating $x_1,x_2$ again yields a degree~$6$ polynomial, valid uniformly for all~$n$.
\end{remark}

\subsection{Planar Vandermonde cells}\label{subsec:Planar Vandermonde cells}

We now specialize to the planar case $d=3$, where resultants give closed-form boundary equations.

\begin{lemma}\label{Equations planar boundaries}
The boundary of $\Pi_{n,3}$ consists of segments of the curves $b_2,\dots,b_n$, where
\[
b_k =
x^3(k^3 - 4k^2 + 4k)
+ y^2(-k^3 + k^2)
+ xy(6k^2 - 6k)
+ x^2(-3k^2 + 3k - 3)
+ 3kx
+ y(-4k + 4)
- 1.
\]
Here $b_2$ is a line, while for $k>2$ each $b_k$ is a cuspidal cubic.
\end{lemma}

\begin{proof}
From the parametrization in Lemma~\ref{lem:Parametrisation for all d}, there is exactly one new boundary curve when $n$ increases.  
To obtain its defining equation, we use the \emph{resultant} of the corresponding rational parametrization.

Given a rational parametrization
\[
t \longmapsto \bigl(x(t)=P_x(t)/Q(t),\,y(t)=P_y(t)/Q(t)\bigr),
\]
a defining equation of the image is given by the resultant
\[
p(x,y) = \mathcal{R}\bigl(P_x(t) - xQ(t),\, P_y(t) - yQ(t)\bigr),
\]
where $\mathcal{R}(A,B)$ denotes the resultant of $A$ and $B$, equivalently $\det(\Phi_{A,B})$ for the linear map $\Phi_{A,B}(C,D) = AC + BD$, with $\deg(C) = \deg(B)$ and $\deg(D) = \deg(A)$; see~\cite[§7.1.2, p.~40]{arkani2017positive}.

In our case,
\[
Q(t) = (n-1)^2, \quad
P_x(t) = (1-t)^2(n-1) + t^2(n-1)^2, \quad
P_y(t) = (1-t)^3 + t^3(n-1)^3.
\]
A presenting matrix for $\Phi_{P_x(t)-xQ(t),\,P_y(t)-yQ(t)}$ in the standard monomial basis is
\[
\begin{pmatrix}
\frac{(n-1)^2-1}{(n-1)^2} & 0 & \tfrac{n}{n-1} & 0 & 0 \\
\tfrac{3}{(n-1)^2} & \tfrac{(n-1)^2-1}{(n-1)^2} & \tfrac{-2}{n-1} & \tfrac{n}{n-1} & 0 \\
\tfrac{-3}{(n-1)^2} & \tfrac{3}{(n-1)^2} & \tfrac{1}{n-1}-x & \tfrac{-2}{n-1} & \tfrac{n}{n-1} \\
\tfrac{1}{(n-1)^2}-y & \tfrac{-3}{(n-1)^2} & 0 & \tfrac{1}{n-1}-x & \tfrac{-2}{n-1} \\
0 & \tfrac{1}{(n-1)^2}-y & 0 & 0 & \tfrac{1}{n-1}-x
\end{pmatrix}.
\]
Taking its determinant gives
\[
\frac{1}{(n-1)^4}\Bigl(
n\bigl(n^3(x^3-y^2)+n^2(-4x^3-3x^2+6xy+y^2)
+n(4x^3+3x^2+3x-6xy-4y)-3x^2+4y-1\bigr)\Bigr),
\]
which simplifies to
\[
\frac{-1}{(n-1)^4}\Bigl(
x^3(n^4-4n^3+4n^2)
+y^2(-n^4+n^3)
+xy(6n^3-6n^2)
+x^2(-3n^3+3n^2-3n)
+3n^2x
+y(-4n^2+4n)
-n
\Bigr).
\]
Normalizing so that the constant coefficient is~$1$ yields the claimed family of equations~$b_k$.

For $k=2$ we observe that $b_2$ factors as the square of a linear form, hence defines a line.  
By L\"uroth’s theorem, all unirational curves are rational, so $\{b_k=0\}$ is a rational curve for all $k \ge 2$.  
A direct computation shows that the singular locus of $b_k$ consists of the single point $(1/k,\,1/k^2)$, which corresponds to a unique preimage under the parametrization; hence the singularity is a cusp.
\end{proof}

\section{Positive geometry and canonical forms}\label{sec:positive_geometry}
\noindent
Building on the boundary description developed in Section~\ref{sec:vandermonde-cells}, we now reinterpret Vandermonde cells as instances of positive geometries and analyze their associated canonical differential forms. 
In this section we extend the framework of positive geometries to include semi-algebraic regions whose interiors may contain singularities, showing that planar Vandermonde cells naturally fall within this broader setting (Definition~\ref{DefPositiveGeometry}). 
Using the additivity property of canonical forms (Proposition~\ref{AdditivityofCanForms}), we compute explicit canonical forms for planar Vandermonde cells by subdividing them into elementary regions bounded by lines and cuspidal cubics (Propositions~\ref{Prop:SemiAlgTypes} and~\ref{SubdivisionOfPlanarCells}). 
Explicit canonical forms for the resulting configurations are derived in Lemmas~\ref{lem:typeI}--\ref{lem:typeIII}, and their logarithmic nature is verified in Lemma~\ref{lem:typeI}. 
Finally, we confirm that these forms satisfy residue recursion and hence define  positive geometries (Theorem~\ref{Thm:PlanarCellsarePositiveGeometries}). 

\smallskip
We begin with general preliminaries before turning to explicit computations in subsequent subsections.

\subsection{Preliminaries: logarithmic forms, residues, and functoriality}
\label{subsec:log-forms}
This subsection recalls the analytic and geometric tools required for constructing canonical forms of Vandermonde cells. 
We review logarithmic differential forms, residues, and functoriality, following~\cite{arkani2017positive}. 

\medskip\noindent\textbf{Affine and projective setup.}
We regard $\RR^n$ as the affine chart $\{x_{n+1}\neq 0\}$ inside projective space 
$\PP^n_{\RR} := \{(x_1:\dots:x_{n+1}) \mid (x_1,\dots,x_{n+1})\neq 0\}$ 
and denote the coordinates by $x, y, \ldots$. 
Throughout, we move freely between affine and projective viewpoints.

\medskip\noindent\textbf{Pullbacks and pushforwards.}
Let $X$ be an irreducible complex variety of dimension $n$, and let $Y\subset X$ be a codimension-one subvariety such that $X\setminus Y$ is smooth. 
In our applications, the ambient varieties $X$ may be singular. 
For a holomorphic map $f$, the pullback $f^*(\omega)$ of a differential form is defined by precomposition. 
If $f$ is a proper and surjective morphism between varieties of the same dimension, we define the pushforward (or trace) $f_*(\omega)$ as the sum of pullbacks along the finitely many local inverses of $f$~\cite[II~(b)]{Griffiths1976Variations}. 
In particular, resolutions of singularities are proper and surjective, so pushforwards along them are well defined.

\medskip\noindent\textbf{Meromorphic and holomorphic forms.}
By a rational or meromorphic differential form $\omega$ on $X$ we mean a rational form defined on the smooth locus $X_{\mathrm{reg}}\subset X$. 
We say that $\omega$ is \emph{holomorphic} if it pulls back to a holomorphic form $\rho^{*}(\omega)$ on a smooth model $\tilde{X}$ under a resolution of singularities $\rho\colon \tilde{X} \to X$. 

\medskip\noindent\textbf{Extension over singularities.}
Determining when the pullback of a holomorphic form on $X_{\mathrm{reg}}$ extends across the exceptional locus of a resolution is a subtle question. 
A classical theorem of Kempf asserts that if $X$ has rational singularities, then every form obtained as the pullback of a holomorphic form on $X_{\mathrm{reg}}$ extends holomorphically over the exceptional locus; see~\cite{ArapuraHolomorphicForms} for a comprehensive discussion of holomorphic forms on singular varieties. 

\medskip\noindent\textbf{Top forms and the cotangent sheaf.}
We denote by $\Omega_X^n$ the top exterior power of the cotangent sheaf $\Omega_X$ of $X$, which is locally free if and only if $X$ is smooth. 
A global section of $\Omega_X^n$ restricts to a holomorphic differential form on the smooth locus $X_{\mathrm{reg}}$, and its pullback along a resolution $\rho$ defines a global section of the line bundle $\Omega^n_{\tilde{X}}$. 
In particular, any non-zero holomorphic $n$-form on $X$ pulls back to a non-zero section of $\Omega_{\tilde{X}}^n$, and conversely, any non-zero section of $\Omega_{\tilde{X}}^n$ pushes forward to a holomorphic form on $X$. 
This correspondence allows us to pass between singular and resolved models without losing information about holomorphic top forms.

\medskip\noindent\textbf{Logarithmic differential forms.}
We now recall the notion of a logarithmic form (see {\cite[§7]{deligne1970equations}, \cite[§8]{weber2005residue}}). 
Let $X$ be an irreducible variety and $Y \subset X$ a hypersurface.  
A \emph{logarithmic} $n$-form $\omega$ on $X$ with poles along $Y$ is a rational differential form on $X$ satisfying the following conditions:
\begin{itemize}
    \item If $X$ is smooth, then $\omega$ has at most simple poles along $Y$; that is, for any local equation $f$ of $Y$, the form $f\omega$ extends holomorphically across $Y$.  
    \item If $X$ is singular, choose an embedded resolution $\rho\colon\tilde{X}\to X$ such that $\rho^{-1}(Y)$ is a simple normal crossings divisor.~Then $\omega$ has at most simple poles along $Y$ if $\rho^*(\omega)$ has at most simple poles along~$\rho^{-1}(Y)$.  
    \item If $Y$ is a simple normal crossings divisor locally given by $\{f_1\cdots f_r=0\}$, then $\omega$ is logarithmic if it can be written locally as
      $\omega = \eta \wedge \frac{df_1}{f_1}\wedge \dots \wedge \frac{df_r}{f_r}$,
    where $\eta$ is holomorphic.  
    \item For general $Y$, we call $\omega$ logarithmic if $\rho^*(\omega)$ is logarithmic for some resolution $\rho$ as above.
\end{itemize}

\medskip\noindent\textbf{Spaces of logarithmic forms.}
We denote by $\Omega^n_{\log}(X\setminus Y)$ the space of logarithmic $n$-forms on $X$ with poles along $Y$.  
For a simple normal crossings divisor $Y=\{f_1\cdots f_r=0\}$, every form of the type 
$\eta \wedge \frac{df_1}{f_1}\wedge \dots \wedge \frac{df_r}{f_r}$,
with $\eta$ holomorphic, belongs to $\Omega^n_{\log}(X\setminus Y)$.  
For example, if $X=\PP^n$ and $Y=\{g=0\}$ is a hypersurface of degree $d>n$, then the forms with simple poles along $Y$ are precisely those of the type
$\frac{f}{g}\, dx$, where $f$ is homogeneous of degree $d-n-1$. 
If $Y$ is a simple normal crossings divisor, every such form has logarithmic poles along $Y$ and thus belongs to $\Omega^n_{\log}(X\setminus Y)$. 
Finally, if $Y$ is not a simple normal crossings divisor, there exist differential forms on $X\setminus Y$ with simple poles along $Y$ that are not logarithmic; see Example~\ref{Example:non-log Pole}.

\medskip\noindent\textbf{Poincaré residue.}
Given a logarithmic form $\omega$ on $X$ with poles along a hypersurface $Y$, one can associate to each irreducible component of $Y$ a lower-dimensional differential form called its \emph{residue}. 
This construction, originating with Poincaré and formalized in modern treatments such as~\cite{weber2005residue}, plays a central role in the recursive structure of canonical forms.
More precisely, let 
$\omega \in \Omega_{\log}(X\setminus Y)$ be a logarithmic differential form.  
For a prime divisor $D \subset Y$, the \emph{residue} of $\omega$ along $D$ is locally defined by
$\operatorname{Res}_D(\omega) := \eta$,
where $\omega = \tfrac{df}{f}\wedge \eta + \eta'$ and $f$ is a local equation of $D$.

\noindent
Residues can be computed on an embedded resolution of singularities 
$\rho\colon (\tilde{X},\rho^{-1}(Y))\to(X,Y)$, where $\rho^{-1}(Y)$ is a simple normal crossings divisor, 
and then pushed forward to $(X,Y)$. 
In particular, the residue along a singular hypersurface $Y$ is well defined as a differential form on its smooth locus $Y^{\mathrm{reg}}$.

\medskip

The following result from~\cite{brown2025positive} provides a criterion for verifying that a form is logarithmic. 

\begin{proposition}[\cite{brown2025positive}, Prop.~1.15]\label{Prop:logform}
Let $X$ be a complex variety and $Y\subset X$ a hypersurface. 
For a differential form $\omega$ on $X$ with simple poles along $Y$, we have
\[
\omega \in \Omega_{\log}(X\setminus Y) 
\iff 
\operatorname{Res}_D(\omega) \in \Omega_{\log}(Y^{\mathrm{reg}}\cap D)
\quad\text{for all prime divisors } D\subset Y.
\]
\end{proposition}

The criterion in Proposition~\ref{Prop:logform} shows that logarithmicity can fail when the pole divisor is singular. 
The following example illustrates this phenomenon for a cuspidal cubic in the plane.

\begin{example}\label{Example:non-log Pole}
Consider the form $\omega=\tfrac{dx\wedge dy}{y^2-x^3}$ on $\PP^2$, with pole divisor $Y=\{y^2-x^3=0\}$. 
The form $\omega$ has a simple pole along $Y$. 
By Proposition~\ref{Prop:logform}, $\omega$ is logarithmic if and only if its residue $\operatorname{Res}_{\{y^2=x^3\}}(\omega)=\tfrac{dy}{-3x^2}=\tfrac{dx}{-2y}$ on the cubic $Y$ is logarithmic. 
This, in turn, requires that $\rho^*(\operatorname{Res}_{\{y^2=x^3\}}(\omega))$ have at most simple poles along the singular locus of $Y$, where $\rho\colon\CC\to\{y^2=x^3\}$, $t\mapsto(t^2,t^3)$, is the embedded resolution in the $(x,y)$-chart. 
We compute $\rho^*(\operatorname{Res}_{\{y^2=x^3\}}(\omega))=\tfrac{dt}{t^2}$, which has a pole of order two at $t=0$; hence $\omega$ is not logarithmic. 
The failure of logarithmicity arises from the cusp singularity of the non–simple normal crossings divisor $Y$, where the residue acquires a higher-order pole.
\end{example}

\subsection{Positive geometry}
\label{subsec:positive-geometry}

Building on the notation and definitions from Section~\ref{subsec:log-forms}, we now recall the notion of a positive geometry.

\begin{definition}[Positive geometry {\cite{arkani2017positive}}]\label{DefPositiveGeometry}
A \emph{positive geometry} of dimension $d$ is a pair $(X,X_{\geq 0})$, where $X$ is a $d$-dimensional irreducible complex projective variety defined over $\RR$, and $X_{\geq 0}\subset X(\RR)$ is a nonempty, closed, full-dimensional semialgebraic subset of the real locus $X(\RR)$. 
We require the existence of a unique rational differential form $\omega$ on $X$ with logarithmic poles precisely along the Zariski closure of the topological boundary $\partial X_{\geq 0}$.  
For each prime boundary divisor $D\subset \partial X:=\overline{X_{\geq 0}}$, we set $D_{\geq 0}:=D\cap X_{\geq 0}\subset D(\RR)$. 
The residue $\operatorname{Res}_D(\omega)$ then equips the pair $(D,D_{\geq 0})$ with the structure of a positive geometry. 
At dimension zero, a point is a positive geometry with canonical form $\omega=\pm 1$.
\end{definition}

\begin{remark}[On regularity assumptions]\label{Rem:Regularity}
The original formulation of positive geometries in~\cite{arkani2017positive} imposed stronger regularity conditions: the interior $X_{\geq 0}^{\circ}$ of the nonnegative part $X_{\geq 0}$ was required to be an open, oriented $d$-dimensional manifold whose closure equals $X_{\geq 0}$, and the ambient variety $X$ was assumed to be normal. 
Under these hypotheses, planar Vandermonde cells do not qualify as positive geometries in the strict sense, since their boundaries contain a cuspidal cubic whose cusp lies in the topological interior of that boundary component.  

Later work has relaxed these conditions. 
In particular, non-normal positive geometries $(X,X_{\geq 0})$ satisfying $(X\setminus X_{\mathrm{reg}})\cap X_{\geq 0}^{\circ}=\emptyset$ have been studied previously. 
In the present work, we consider the stronger case where the intersection $(X\setminus X_{\mathrm{reg}})\cap X_{\geq 0}^{\circ}$ may be nonempty.
\end{remark}

\noindent The next examples show how positive geometries go beyond the regular setting of~\cite{arkani2017positive}.

\begin{example}\label{Ex: Positive geometries}
    \begin{itemize}
        \item \emph{Simplex.} 
        An $n$-dimensional simplex $\mathcal{S}:=\{(x_1,\dots,x_n)\mid 0\leq x_i, \sum_{i=1}^nx_i=1\}\subset \RR^n\subset \PP^n$ is a positive geometry with canonical form
        \[
        \omega_{\mathcal{S}}=\frac{d\boldsymbol{x}}{(1-\sum_{i=1}^nx_i)\prod_{i=1}^{n}x_i}.
        \]
        
        \item \emph{Non-normal boundary with smooth interior} (\cite{arkani2017positive}, Example~5.3.1). 
        For $a\neq 0$, the teardrop region enclosed by the nodal cubic $D=\{y^2 \leq x^2(x+a^2)\}$ has canonical form
        \[
        \omega_D=\frac{dx\wedge dy}{y^2-x^2(x+a^2)}.
        \]
        
        \item \emph{Non-normal boundary with smooth interior} (\cite[Fig.~4, Remark~2.12(2)]{kohn2025adjoints}.
        The semi-algebraic region $E :=\{y^2\leq x^3 , x\geq y^2\}$ bounded by a cuspidal cubic and a parabola has canonical form
        \[
        \omega_E=\frac{x^2+y^2+xy+y}{(x-y^2)(y^2-x^3)}\,dx\wedge dy.
        \]
    \end{itemize}
\end{example}

\medskip
\noindent\textbf{Non-normal boundaries.}
Non-normal positive geometries were already discussed in~\cite[Example~5.3.1]{arkani2017positive} and~\cite[Remark~2.12(2)]{kohn2025adjoints}; see Examples~\ref{Ex: Positive geometries}(2)--(3). 
In these cases, singularities of irreducible boundary components—necessarily non-normal for singular curves—do not lie in the interior of the boundary but form zero-dimensional strata.
Smooth interiors are essential for structural results, such as the characterization of certain planar positive geometries as \emph{quasi-regular rational polypols} in~\cite[Proposition~2.11]{kohn2025adjoints}.  

\medskip
\noindent\textbf{Singular interiors.}
In our setting, we also allow semi-algebraic sets whose interior $X_{\geq 0}^{\circ}$ is not smooth. 
This relaxation includes additional examples of positive geometries, most notably the planar Vandermonde cells. 
To work with singular boundaries, we fix a notion of holomorphic and meromorphic forms on singular $X$ via their pullbacks to smooth models.

\begin{example}[The first nontrivial Vandermonde cell]
The Vandermonde cell $\Pi_{3,3}$ is the semi-algebraic region bounded by the line $b_2$ and the cuspidal cubic $b_3$. 
Since the cusp lies in the interior of the cubic, this region is neither orientable nor a manifold. 
Nevertheless, planar Vandermonde cells are positive geometries in the sense of Definition~\ref{DefPositiveGeometry}.
\end{example}

A key structural property of canonical forms, \emph{additivity}, will play a central role in our later analysis of Vandermonde cells; see~\cite[\S3]{arkani2017positive} for a general discussion.
For a positive geometry $(X,X_{\geq 0})$, we say that a collection $(X,X_{\geq 0,1}),\dots,(X,X_{\geq 0,n})$ \emph{subdivides} it if $\bigcup_i X_{\geq 0,i}=X_{\geq 0}$ and the interiors $X_{>0,i}$ are pairwise disjoint. 
Unlike classical polytopal subdivisions, we do not require that pairwise intersections of the $X_{\geq 0,i}$ form boundary components of both.

\begin{proposition}
\label{AdditivityofCanForms}
Let $(\PP^d,P_{1}),\dots,(\PP^d,P_{n})$ be projective polytopes subdividing the projective polytope $(\PP^d,P)$, oriented so that the induced orientations on their boundaries are compatible. Then
\[
\Omega_P=\sum_{i}\Omega_{P_i}.
\]
\end{proposition}

The proof of Proposition~\ref{AdditivityofCanForms} reduces to the case of simplices, though additional subtleties arise in the non-polytopal setting. 
Additivity also extends beyond the polytope case: it allows one to compute canonical forms of semi-algebraic regions by approximating them with convergent sequences of polytopal subdivisions; see~\cite[§10]{arkani2017positive} and Lemma~\ref{ConvexApproximation}. 
In later sections, we will encounter the formulation of positive geometries in~\cite{brown2025positive}, where triangulation-independence follows directly from the construction.

\medskip
We conclude this subsection with a remark on the internal consistency of this framework and on the structural constraints implied by the uniqueness of canonical forms.

\begin{remark}[Additivity and vanishing genus]
Allowing singularities in the interior of $X_{\geq 0}$ remains compatible with the additivity principle stated above, provided that the compatibility conditions are suitably weakened. 
For example, in Lemma~\ref{lem:typeI} the semi-algebraic set of type~I can be subdivided into two regions whose boundaries are smooth, so their canonical forms can be computed as in~\cite{kohn2025adjoints}; additivity then yields the canonical form of the original set.

Moreover, the uniqueness of the canonical form $\omega$ of a positive geometry $(X,X_{\geq 0})$ implies that $X$ admits no nonzero holomorphic forms, since multiplication by such a form would preserve all defining properties. 
Consequently, $X$ and every boundary hypersurface $Y$ appearing in the boundary recursion must have vanishing geometric genus. 
In Theorem~\ref{Thm:Vanishing geometric genus} we use Proposition~\ref{prop:Unirational} to show that the boundary hypersurfaces of Vandermonde cells indeed have genus~$0$, where genus denotes the invariant introduced in~\cite{brown2025positive} generalizing the geometric genus for smooth varieties.
\end{remark}

Having established the general framework, we now apply it to Vandermonde cells.

\subsection{Canonical forms of planar Vandermonde cells via subdivision}

In this subsection we determine canonical forms of planar Vandermonde cells, thereby extending the class of known planar positive geometries beyond polypols.  
Our approach relies on the additivity property of canonical forms (Proposition~\ref{AdditivityofCanForms}), which allows us to compute the canonical form of a Vandermonde cell by decomposing it into simpler regions whose canonical forms are easier to determine.

\medskip
\noindent
This additive structure provides the foundation for our explicit constructions.  
In the planar case, each Vandermonde cell can be subdivided into regions bounded by rational curves, so that its canonical form is obtained as the sum of the canonical forms of these elementary pieces.  
We now describe the three basic configurations that appear in such subdivisions.

\medskip
By Lemma~\ref{Equations planar boundaries}, the irreducible boundary components $\{b_k=0\}_{3 \leq k\leq n}$ of $\Pi_{n,3}$ are cuspidal cubics. 
Up to a linear change of coordinates, each such cubic is given by $y^2=x^3$. 
Since canonical forms are compatible with coordinate changes via pullback and pushforward, all computations can be performed in this normal form; canonical forms in other coordinates then follow by pullback, pushforward, or geometric arguments, as degrees and incidences are preserved under linear transformations.

\medskip
All semi-algebraic regions arising in the subdivision of planar Vandermonde cells (see \S\ref{SubdivisionOfPlanarCells}) are bounded by a cuspidal cubic and a line.  
Their configurations fall into the following three types:
\begin{enumerate}
    \item[{\rm (I)}] the cubic and the line intersect in three distinct points;  
    \item[{\rm (II)}] they intersect in two distinct points, with one intersection of multiplicity two away from the cusp;  
    \item[{\rm (III)}] they intersect in two distinct points, with one intersection of multiplicity two at the cusp.  
\end{enumerate}
Explicit realizations and canonical forms for these three configurations are provided in Lemmas~\ref{lem:typeI}, \ref{lem:typeII}, and~\ref{lem:typeIII}. 
The following proposition shows that, after an appropriate projective change of coordinates, any semi-algebraic region bounded by a cuspidal cubic and a line whose cusp lies on the topological boundary can be reduced to one of these three types. 
Hence it suffices to determine canonical forms for these representative examples. 
Canonical forms of types~I and~III also appear in~\cite[Propositions~5.3,~5.5.2]{brown2025positive}.

\begin{proposition}\label{Prop:SemiAlgTypes}
Let $A \subset \RR^2 \subset \PP^2$ be a compact, connected, planar semi-algebraic set bounded by a line and a cuspidal cubic. 
Assume that the cusp and exactly two vertices (intersection points of the boundary curves where the boundary switches) lie on the topological boundary of $A$. 
Then, up to a projective coordinate change $g\in {\rm Aut}(\PP^2)\cong {\rm PGL}(3)$, the region $A$ is equivalent to one of the following sets:
\begin{equation}\label{eq:semi_alg_types}
A_{I}=\{\,y^2<x^3,\;x<1\,\},\  
A_{II}=\{\,y^2<x^3,\;y-\tfrac{3}{2}x+\tfrac{1}{2}>0,\;y<1\,\},\ 
 A_{III}=\{\,y^2<x^3,\;y<x,\;y>0\,\}.
\end{equation}

\end{proposition}
\begin{figure}[htbp]
\centering

\begin{minipage}{0.2\textwidth}
  \centering
  \includegraphics[width=\linewidth]{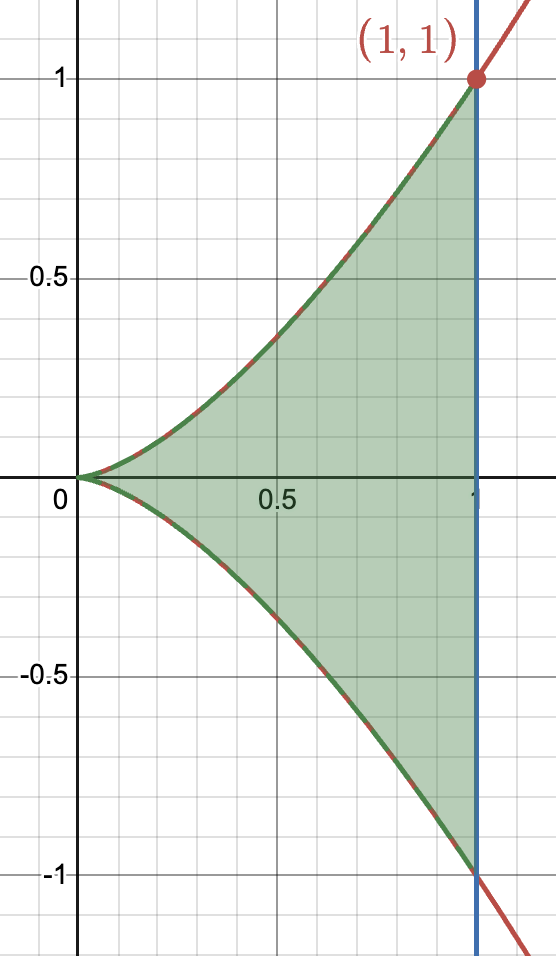}
  {\small  $A_I$}
\end{minipage}
\hfill
\begin{minipage}{0.2\textwidth}
  \centering
  \includegraphics[width=\linewidth]{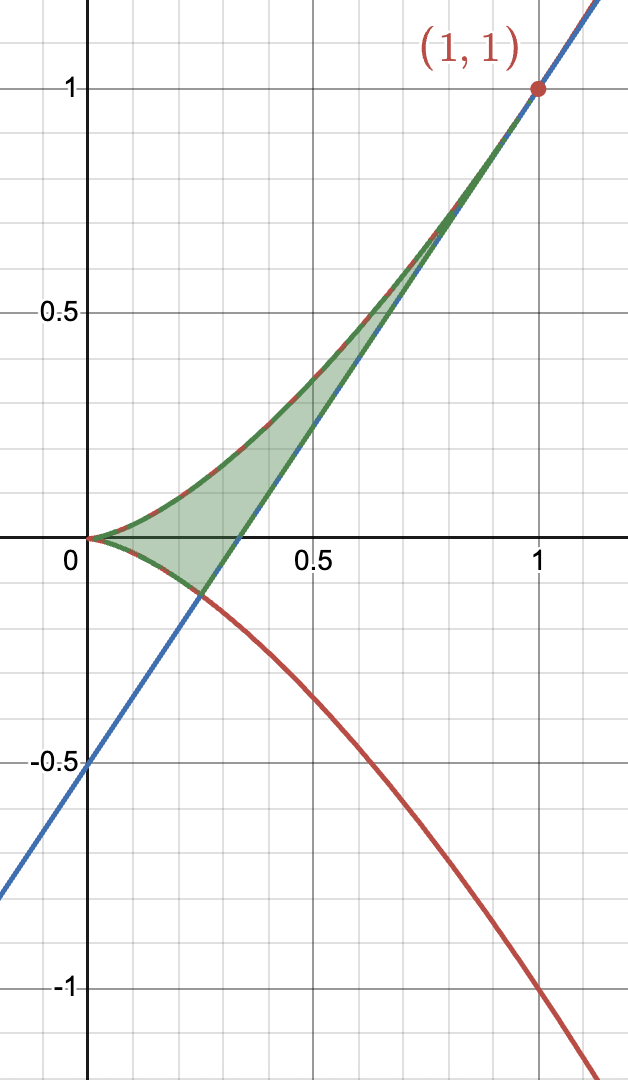}
  {\small $A_{II}$}
\end{minipage}
\hfill
\begin{minipage}{0.2\textwidth}
  \centering
  \includegraphics[width=\linewidth]{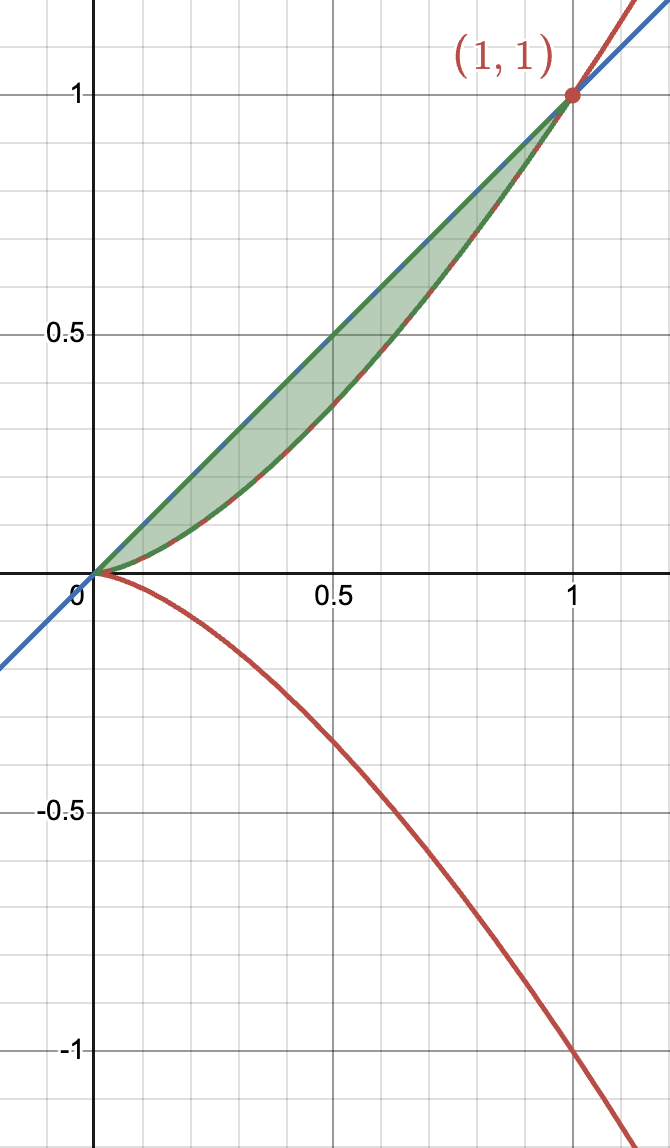}
  {\small $A_{III}$}
\end{minipage}

\caption{Semialgebraic sets arising in the subdivision of the planar Vandermonde cell from Example~\ref{Ex:GeometricArgumentforFirstCanForm}.}
\label{fig:planar}
\end{figure}

\begin{proof}
The automorphism group ${\rm Aut}(\PP^2)={\rm PGL}(3)$ acts $4$-transitively on the points of $\PP^2$. 
Cuspidal cubics are preserved under projective transformations, as they are precisely the cubic curves admitting a Weierstrass equation of the form $y^2=p(x)$ with $p(x)$ a monic cubic possessing a triple root.  

Let $A$ be bounded by a cuspidal cubic $C$ and a line $L$. 
Since $A$ is compact, the line and the cubic meet in exactly two distinct real points, implying that the third intersection point, counted with multiplicity, is also real. 
Up to projective equivalence, the configuration $(C,L)$ falls into one of the three cases below.

\medskip
\textbf{Case~1.} If $L\cap C=\{p_1,p_2,p_3\}$ consists of three distinct real points, choose $g\in {\rm PGL}(3)$ sending the cusp to $(0\!:\!0\!:\!1)$ and the intersection points to $(1\!:\!0\!:\!1)$, $(0\!:\!1\!:\!1)$, and $(0\!:\!1\!:\!0)$. 
In these coordinates, the unique connected component of $\PP^2\setminus(L\cup C)$ whose boundary contains the cusp and two intersection points is $A_I$.  

\medskip
\textbf{Case~2.} Suppose $L\cap C=\{p_1,p_2\}$, where $p_1$ has multiplicity~$2$ and is distinct from the cusp. 
Send the cusp to $(0\!:\!0\!:\!1)$, the double point $p_1$ to $(1\!:\!1\!:\!1)$, and the simple intersection $p_2$ to $(\tfrac{1}{4}\!:\!\tfrac{1}{8}\!:\!1)$. 
In this normalization, the region bounded by $L$ and $C$ with two vertices corresponds to $A_{II}$.  

\medskip
\textbf{Case~3.} If the non-transversal intersection coincides with the cusp, map the cusp to $(0\!:\!0\!:\!1)$ and the remaining intersection point $p_2$ to $(1\!:\!1\!:\!1)$. 
The unique connected component with two vertices in this configuration is $A_{III}$.  
\end{proof}


\begin{remark}
   \begin{itemize}
       \item In  the following proposition, 
       we construct a subdivision of the planar Vandermonde cell $\Pi_{n,3}$ such that each component region is bounded precisely by a cuspidal cubic and a line. 
By Proposition~\ref{Prop:SemiAlgTypes}, every such region is projectively equivalent to one of the semi-algebraic types described in \eqref{eq:semi_alg_types}. 
Consequently, the canonical form of $\Pi_{n,3}$ is obtained by summing the canonical forms of these constituent regions. (We will explicitly compute these canonical forms in Lemmas~\ref{lem:typeI} and~\ref{lem:typeII}.)
Moreover, each semi-algebraic region arising in the subdivision has total boundary degree~$4$, so its canonical form is uniquely determined by the \emph{linear adjoint polynomial}.  
Because a linear adjoint is specified by any two distinct points lying on it, it can be characterized through two incidence relations; see Example~\ref{Ex:GeometricArgumentforFirstCanForm} for the explicit calculation. 
   \end{itemize}
\end{remark}

\begin{proposition}\label{SubdivisionOfPlanarCells}
For every $n\ge 3$, the planar Vandermonde cell $\Pi_{n,3}$ admits a decomposition into $n-3$ semi-algebraic sets of type~$A_{II}$ 
and $n-2$ semi-algebraic sets of type~$A_{III}$ as in \eqref{eq:semi_alg_types}. 
\end{proposition}

\begin{proof}
By Lemma~\ref{Equations planar boundaries}, passing from $\Pi_{n,3}$ to $\Pi_{n+1,3}$ introduces a single new boundary component $b_{n+1}=0$. Using the parametrizations from Lemma~\ref{lem:Parametrisation for all d}, one sees that the new portion is the bounded region of the semi-algebraic set 
\[
\; \{\, b_{n+1}>0,\; b_n<0 \,\}.
\]
The boundary curves of this region meet at exactly two vertices: the point $(1,1)$ and the cusp $c_n$ of $\{b_n=0\}$.

Let $\tilde{c}_n=0$ denote the defining equation of the unique line through $c_n$ having triple contact (intersection multiplicity $3$) with $\{b_n=0\}$ at $c_n$. Then
\[
\Pi_{n+1,3}\setminus \Pi_{n,3}
\;=\;
\{\, b_{n+1}>0,\; b_n<0,\; \tilde{c}_n<0 \,\}.
\]
We further subdivide this region by the line $l(c_n,(1,1))$ through $c_n$ and $(1,1)$. It meets $\{b_n=0\}$ with multiplicity $2$ at $c_n$ and multiplicity $1$ at $(1,1)$, so there are no additional intersections. Likewise, $l(c_n,(1,1))$ meets $\{b_{n+1}=0\}$ with multiplicity $1$ at $c_n$ and, non-transversally, with multiplicity $2$ at $(1,1)$. Consequently, $l(c_n,(1,1))$ divides the simply connected region $\Pi_{n+1,3}\setminus \Pi_{n,3}$ into two simply connected components, each bounded by a line and a cuspidal cubic but with distinct tangency behavior. By Proposition~\ref{Prop:SemiAlgTypes}, these are precisely the semi-algebraic types~I and~II.

For the base case, $\Pi_{3,3}$ is itself of type~II (bounded by $\{b_2=0\}$ and the cuspidal cubic $\{b_3=0\}$ meeting with multiplicity $2$ at $(1,1)$). At each step $k\mapsto k+1$ ($k\ge 3$), we add exactly one new region of type~I and one of type~II. Therefore, by induction, for $\Pi_{n,3}$ (after $n-3$ steps) the total counts are
\[
\#\text{(type II)} = n-3,\qquad
\#\text{(type III)} = 1 + (n-3) = n-2,
\]
as claimed.
\end{proof}

We begin the canonical-form analysis of $A_I$ from \eqref{eq:semi_alg_types} in Proposition~\ref{Prop:SemiAlgTypes}. 
This can be viewed as the first non-trivial 
Vandermonde cell with the non-transversal intersection replaced by a transversal one.

\begin{lemma}[Semi-algebraic set $A_I$]\label{lem:typeI}
The planar semi-algebraic set $A:=\{(x,y)\mid x\le 1,\; y^2\le x^3\}$ bounded by the line $x=1$ and the cuspidal cubic $y^2=x^3$ is a positive geometry. 
\begin{itemize}
    \item[{\rm {\bf(a)}}] Its canonical form is
\begin{equation}\label{eq:A_I}
\omega_A=\frac{-2x}{(x-1)(y^2-x^3)}\,dx\wedge dy.
\end{equation}
 \item[{\rm {\bf(b)}}] Moreover, 
$\omega_{A}$ 
is logarithmic, and its residue along the cuspidal cubic $C=\{y^2=x^3\}$ agrees with the pushforward of the canonical form $\omega_{[-1,1]}$ of the interval $[-1,1]$ under the normalization
\[
\gamma\colon \mathbb{P}^1 \to C \subset \mathbb{P}^2,\qquad (t:1)\longmapsto (t^2:t^3:1).
\]
\end{itemize}
\end{lemma}

\begin{proof}
{\bf(a)} Decompose $A$ into two polypols $S_1$ (with $y>0$) and $S_2$ (with $y<0$), each bounded by $x=1$, $y=0$, and $y^2=x^3$. For each $S_i$, the canonical form has a degree-$2$ numerator; we take the ansatz
\[
\omega_{S_i}=\frac{a_ix^2+b_iy^2+c_ixy+d_ix+e_iy+f_i}{y(x-1)(y^2-x^3)}\,dx\wedge dy.
\]
Matching the boundary residues along $y=0$ fixes, for $S_1$, $a_1=1$ and $d_1=f_1=0$. Matching along $x=1$ further gives $b_1=0$ and $c_1+e_1=1$, hence
\[
\omega_{S_1}=\frac{x^2+c_1xy+e_1y}{y(x-1)(y^2-x^3)}\,dx\wedge dy.
\]
By symmetry across $y\mapsto -y$, the corresponding form on $S_2$ is
\[
\omega_{S_2}=\frac{-x^2+c_2xy+e_2y}{y(x-1)(y^2-x^3)}\,dx\wedge dy,
\qquad c_2+e_2=1.
\]
Summing,
\[
\omega_A=\omega_{S_1}+\omega_{S_2}
=\frac{(c_1+c_2)xy+(e_1+e_2)y}{y(x-1)(y^2-x^3)}\,dx\wedge dy.
\]
Logarithmicity along the cubic is equivalent to the adjoint passing through the cusp $(0\!:\!0\!:\!1)$, which imposes $e_1+e_2=0$ and hence $c_1+c_2=2$. Substituting yields
\[
\omega_A=\frac{-2x}{(x-1)(y^2-x^3)}\,dx\wedge dy.
\]

{\rm {\bf(b)}} Our main tool to show the logarithmic property along $y^2=x^3$ is Proposition~\ref{Prop:logform} together with the explicit resolution of the cuspidal cubic. Due to proposition~\ref{Prop:logform} we only need to check if the residue of $\omega_{A_{I}}$ along the cubic $\{y^2-x^3\} $ is logarithmic, which we check via the explicit resolution $\gamma$.
We first compute the residue of $\omega_A$ along $C$. Rewriting
\[
\omega_{A_I}
=\frac{-2x}{(x-1)(y^2-x^3)}\,dx\wedge dy
=\frac{d(y^2-x^3)}{y^2-x^3}\wedge \eta,
\qquad
\eta=\frac{2x}{3x^2(x-1)}\,dy,
\]
gives 
\[
\operatorname{Res}_{\{y^2-x^3=0\}}(\omega_{A_I})
=\left.\frac{2}{3x(x-1)}\,dy\right|_{y^2=x^3}
=\frac{2}{3\sqrt[3]{y^2}\bigl(\sqrt[3]{y^2}-1\bigr)}\,dy,
\]
which is a rational $1$-form on $C_{\mathrm{reg}}=C\setminus\{(0,0)\}$ with poles at $(1,\pm 1)$.

Pulling back this residue along $\gamma$ amounts to substituting $x=t^2$ and $y=t^3$, yielding
\[
\operatorname{Res}_{\{y^2 - x^3 = 0\}}(\omega) 
\frac{2}{3 t^2(t^2-1)}\,d(t^3)
=\frac{2}{(t+1)(t-1)}\,dt
=\omega_{[-1,1]}.
\]
Thus, the pullback coincides with the canonical form of the interval $[-1,1]$ of length $2$. 

Conversely, pushing forward $\omega_{[-1,1]}$ along $\gamma$ yields
\[
\gamma_*\bigl(\omega_{[-1,1]}\bigr)
=\frac{2\,d\sqrt[3]{y}}{\bigl(\sqrt[3]{y}-1\bigr)\bigl(\sqrt[3]{y}+1\bigr)}
=\frac{2}{3\sqrt[3]{y^2}\bigl(\sqrt[3]{y^2}-1\bigr)}\,dy,
\]
which matches $\operatorname{Res}_{\{y^2 - x^3 = 0\}}(\omega)$. Hence $\omega_{A_I}$ is logarithmic, with residue equal to $\gamma_*(\omega_{[-1,1]})$.

Similarly, the canonical form $\omega_{S_1}$ of the region $S_1$ in the proof of {\bf (a)} 
is logarithmic if and only if $e_1=0$. Indeed,
\[
\operatorname{Res}_{\{y^2-x^3=0\}}(\omega_{S_1})
=\left.\frac{x^2+c_1xy+e_1y}{y(x-1)\,3x^2}\,dy\right|_{\{y^2=x^3\}},
\]
so that
\[
\gamma^*\!\left(\operatorname{Res}_{\{y^2-x^3=0\}}(\omega_{S_1})\right)
=\frac{(t^4+c_1t^5+e_1t^3)\,3t^2}{t^3(t^2-1)\,3t^4}\,dt
=\frac{t+c_1t^2+e_1}{t(t^2-1)}\,dt.
\]
This equals $\frac{1}{t(t-1)}\,dt$ if and only if $e_1=0$. The constraint $c_1+e_1=1$ from {\bf (a)} 
then forces $c_1=1$.

We note that this is equivalent to the adjoint line of $S_1$ passing through the cusp, as claimed in the proof of {\bf (a)}. 
\end{proof}

The semi-algebraic set $A_{II}$ can be viewed as a degeneration of $A_I$, where two intersection points coalesce, producing a non-transversal intersection. Since $A_{II}$ is equivalent to the Vandermonde cell $\Pi_{3,3}$ via a projective automorphism $g\in \PP^3$, Lemma~\ref{lem:typeII} suffices to determine its canonical form geometrically.

\begin{lemma}[Semi-algebraic set of type II]\label{lem:typeII}
The planar semi-algebraic set
\[
A_{II} := \{\, y - \tfrac{3}{2}x + \tfrac{1}{2} \geq 0,\; y^2 \leq x^3,\; y \leq 1 \,\},
\]
bounded by the line $y - \tfrac{3}{2}x + \tfrac{1}{2}=0$ and the cuspidal cubic $y^2 = x^3$, is a positive geometry. Its canonical form is
\[
\omega_{A_{II}} = \frac{3}{4}\,\frac{y-x}{(y-\tfrac{3}{2}x+\tfrac{1}{2})(y^2 - x^3)} \, dx \wedge dy.
\]
\end{lemma}
\begin{proof}
The boundary of $A_{II}$ has total degree $4$, so the adjoint curve has degree $1$. We thus have 
\[
\omega_{A_{II}}=\frac{ax+by+c}{\bigl(y-\tfrac{3}{2}x+\tfrac{1}{2}\bigr)\,(y^2-x^3)}\,dx\wedge dy,
\]
for some real parameters $a,b,c$.

To determine these coefficients, we impose residue conditions along the boundary components. Along the line $y-\tfrac{3}{2}x+\tfrac{1}{2}=0$, the residue must match the canonical form of a line segment with endpoints $y=-\tfrac{1}{8}$ and $y=1$, while along the cubic $y^2=x^3$ it must match the pushforward of the canonical form of the interval $[-1,1]$ under the normalization $\gamma$.
For the residue along the line $\{y-\tfrac{3}{2}x+\tfrac{1}{2}=0\}$, we compute
\[
\omega_{A_{II}}
=\frac{ax+by+c}{\bigl(y-\tfrac{3}{2}x+\tfrac{1}{2}\bigr)\,(y^2-x^3)}\,dx\wedge dy
=\frac{-\tfrac{3}{2}\,dx\wedge dy}{\,y-\tfrac{3}{2}x+\tfrac{1}{2}\,}\wedge
\left(-\,\frac{ax+by+c}{-\tfrac{3}{2}\,(y^2-x^3)}\,dx\right),
\]
so
\[
\operatorname{Res}_{\{\,y-\tfrac{3}{2}x+\tfrac{1}{2}=0\,\}}(\omega_{A_{II}})
=-\left.\frac{ax+by+c}{-\tfrac{3}{2}\,(y^2-x^3)}\,dx\right|_{\{\,y-\tfrac{3}{2}x+\tfrac{1}{2}=0\,\}}
=-\frac{ax+b(\tfrac{3}{2}x-\tfrac{1}{2})+c}{\bigl(\tfrac{3}{2}x-\tfrac{1}{2}\bigr)^2-x^3}\,dx
=\frac{(a+\tfrac{3}{2}b)x-\tfrac{1}{2}b+c}{-(x-1)^2\,(x-\tfrac{1}{4})}\,dx.
\]
Comparing to the canonical form $\textstyle{\omega_{[\tfrac{1}{4},1]}=\dfrac{\tfrac{3}{4}}{(t-\tfrac{1}{4})(t-1)}\,dt}$ of an interval of length $\tfrac{3}{4}$, we must have: 
\[
a+\tfrac{3}{2}b=\tfrac{3}{4},\qquad c-\tfrac{1}{2}b=-\tfrac{3}{4}.
\]
With this, the residue along the cubic is
\[
\operatorname{Res}_{\{\,y^2-x^3=0\,\}}(\omega_{A_{II}})
=\left.\frac{(-3c-\tfrac{3}{2})x+(2c+\tfrac{3}{2})y+c}{3x^2\bigl(y-\tfrac{3}{2}x+\tfrac{1}{2}\bigr)}\,dy\right|_{\{\,y^2-x^3=0\,\}},
\]
which pulls back to
\[
\gamma^*\!\left(\operatorname{Res}_{\{\,y^2-x^3=0\,\}}(\omega_{A_{II}})\right)
=\frac{\tfrac{3}{2}t^2+2c\,(t-1)\!\left(t+\tfrac{1}{2}\right)}{t^2(t-1)\!\left(t+\tfrac{1}{2}\right)}\,dt.
\]
Thus $\operatorname{Res}_{\{\,y^2-x^3=0\,\}}(\omega_{A_{II}})$ is the canonical form that pulls back to the canonical form of the segment $\omega_{[-\tfrac{1}{2},1]}$ if and only if $c=0$.
These conditions uniquely determine $a,b,c$ as
\[
a=-\tfrac{3}{4},\qquad b=\tfrac{3}{4},\qquad c=0,
\]
and hence give the stated expression for $\omega_{A_{II}}$.
\end{proof}

To compute canonical forms for planar Vandermonde cells, we use the subdivision in Proposition~\ref{SubdivisionOfPlanarCells}, which naturally produces regions equivalent to $A_{III}$.

\begin{lemma}[Semi-algebraic set $A_{III}$]\label{lem:typeIII}
The canonical form of the bounded semi-algebraic set enclosed by the cuspidal cubic $\{y^2 = x^3\}$ and the line $\{x=y\}$ is
\[
\omega_{A_{III}} = \frac{y}{(y-x)(y^2 - x^3)} \, dx \wedge dy.
\]
\end{lemma}
\begin{proof}
We can obtain $\omega_{A_{III}}$ directly (as for $A_I$ and $A_{II}$), or note that $S_1$ from Lemma~\ref{lem:typeI}, together with $A_{III}$, partitions the triangle $\triangle$ with vertices $(0,0)$, $(1,0)$, and $(1,1)$. Hence the \emph{outer triangulation} gives $\omega_{A_{III}}+\omega_{S_1}=\omega_{\triangle}$. Since
\[
\omega_{\triangle}=\frac{1}{y(x-1)(y-x-1)}\,dx\wedge dy,
\]
substituting this together with $\omega_{S_1}$ from the proof of Lemma~\ref{lem:typeI} yields 
\[
\omega_{A_{III}}
=\omega_{\triangle}-\omega_{S_1}
=\frac{1}{y(x-1)(y-x)}\,dx\wedge dy
-\frac{x^2+xy}{y(x-1)(y^2-x^3)}\,dx\wedge dy
=\frac{y}{(y-x)(y^2-x^3)}\,dx\wedge dy.
\] 
\end{proof}

Having organized the intersection data for all pieces in the subdivision of Proposition~\ref{SubdivisionOfPlanarCells}, we now determine a canonical form for $\Pi_{n,3}$. 
Let $c_k\in\mathbb{R}^2$ denote the cusp of the boundary $\{b_k=0\}$, and for points $A,B\in\mathbb{R}^2$ let $l(A,B)$ be the line through $A$ and $B$. 
We also write $\tilde{c}_n$ for the unique line through $c_n$ that meets $\{b_n=0\}$ with intersection multiplicity $3$.

\begin{theorem}\label{Thm:PlanarCellsarePositiveGeometries}
Planar Vandermonde cells $\Pi_{n,3}$ are positive geometries. 
The canonical form $\Omega_{\Pi_{n,3}}$ is obtained inductively from $\Omega_{\Pi_{n-1,3}}$ by subdividing the additional piece $\Pi_{n,3}\setminus \Pi_{n-1,3}$ with the line $l(c_{n-1},(1,1))$ and adding the two canonical forms
\[
\omega^n_{\mathrm{upper}}=\frac{l(c_n,(1,1))\,dx\wedge dy}{\,l(c_{n-1},(1,1))\,b_n}\quad\text{and}\quad
\omega^n_{\mathrm{lower}}=\frac{\tilde{c}_n\,dx\wedge dy}{\,l(c_{n-1},(1,1))\,b_{n-1}}.
\]
In total,
\[
\Omega_{\Pi_{n,3}}=\Omega_{\Pi_{n-1,3}}+\omega^n_{\mathrm{lower}}+\omega^n_{\mathrm{upper}}.
\]
\end{theorem}
\begin{proof} By Lemmas~\ref{lem:typeII} and~\ref{lem:typeIII}, one has $\Omega_{\Pi_{3,3}}=\frac{l(c_3,(1,1))\,dx\wedge dy}{b_2 b_3}$. In coordinates, the cusp is $c_n=\bigl(\tfrac{1}{n},\tfrac{1}{n^2}\bigr)$, and the line $l(c_n,(1,1))$ is cut out by $-ny+(n+1)x-1$. From Propositions~\ref{SubdivisionOfPlanarCells} and~\ref{Prop:SemiAlgTypes} it follows that increasing $n$ adds exactly one new region of type $A_{II}$ and one of type $A_{III}$ to the previous configuration, starting with a single region of type $A_{II}$ for $\Pi_{3,3}$. These two new regions are obtained by subdividing $\Pi_{n,3}\setminus \Pi_{n-1,3}$ with the line $l(c_{n-1},(1,1))$. Consequently, the canonical form of $\Pi_{n,3}$ is obtained from that of $\Pi_{n-1,3}$ by adding $\omega^n_{\mathrm{upper}}$ and $\omega^n_{\mathrm{lower}}$, as claimed. The presentations of $\omega_{\mathrm{lower}}$ and $\omega_{\mathrm{upper}}$ are easily obtained from Lemma~\ref{lem:typeII} and Lemma~\ref{lem:typeIII}, as explained therein, by arguing via incidences of the linear adjoint curve. 
\end{proof}

\begin{example}\label{Ex:GeometricArgumentforFirstCanForm}
The Vandermonde cell $\Pi_{3,3}$ is bounded by the cuspidal cubic $\{b_3 = 0\}$ and the line $\{b_2 = 0\}$. 
Their intersection at $(1,1)$ is non-transversal. 
By Proposition~\ref{Prop:SemiAlgTypes}, the cell $\Pi_{3,3}$ is projectively equivalent to the region $A_{II}$ described in Lemma~\ref{lem:typeII}. 
Lemma~\ref{lem:typeII} further shows that the canonical form of such a semialgebraic region has an adjoint given by the line passing through the cusp of the cubic and the non-transversal intersection point. 
Accordingly, the adjoint of the first nontrivial planar Vandermonde cell $\Pi_{3,3}$ is the line passing through the cusp $(\tfrac{1}{3}, \tfrac{1}{9})$ of $b_3$ (as defined in Lemma~\ref{Equations planar boundaries}) and the point $(1,1)$, where $\{b_3 = 0\}$ and $\{b_2 = 0\}$ intersect non-transversally. 
See Figure~\ref{fig:planar} for an illustration.
\end{example}

\begin{example}
For $\Pi_{4,3}$, the canonical form is
\begin{align*}
\Omega_{\Pi_{4,3}}
&= \frac{l(c_3,(1,1))}{b_2 b_3}
 + \frac{l(c_4,(1,1))}{l(c_3,(1,1))\,b_4}
 + \frac{\tilde{c}_4}{l(c_3,(1,1))\,b_3}\,dx\wedge dy \\
&= \frac{l(c_3,(1,1))^2 b_4 + l(c_4,(1,1))\,b_2 b_3 + b_2 b_4 \tilde{c}_4}{b_2 b_3 b_4\,l(c_3,(1,1))}\,dx\wedge dy.
\end{align*}
The apparent pole along the line $l(c_3,(1,1))$ is \emph{spurious}; it cancels, although this is not immediately obvious from this presentation.
\end{example}

While the planar case admits explicit canonical forms via subdivision, the situation in higher dimensions is more subtle.

\subsection{Higher-dimensional Vandermonde cells}

\begin{remark}
In Theorem~\ref{Thm:PlanarCellsarePositiveGeometries}, canonical forms of planar Vandermonde cells were computed by subdividing them into full-dimensional pieces whose total boundary degree was strictly smaller than that of the original cell. In contrast, already in the three-dimensional case $\Pi_{4,4}$, the boundary consists of a quadric and a degree~6 hypersurface intersecting in a degree~8 curve. Any hypersurface containing this curve must have degree at least~8. Thus, unlike in the planar case, the total boundary degree does not decrease under analogous subdivisions.
\end{remark}

This illustrates one of the main challenges in constructing canonical forms for higher-dimensional semi-algebraic sets. Existing examples of non-polytopal, non-planar positive geometries in the literature focus on carefully structured cases. For instance, in \cite[§5.6]{brown2025positive} both irreducible boundary components are of low degree with controlled intersections, while in \cite[§5.8]{brown2025positive} the geometry consists of a single hypersurface of degree $n+1$ in $\PP^n$. By contrast, the Vandermonde setting leads to significantly more intricate boundary configurations. Boundary equations can be obtained via elimination techniques using the parametrization from Lemma~\ref{lem:Parametrisation for all d}. See the implementation in \texttt{Macaulay2} at \url{https://github.com/SebSeemann/Boundaries-of-Vandermonde-cells}.

\section{Positive geometries in the sense of Brown–Dúpont}\label{Section:BrownDupont-setup}

The Brown–Dúpont framework~\cite{brown2025positive} reformulates the concept of positive geometries in cohomological terms, replacing recursive definitions by a linear map between relative homology and spaces of logarithmic forms. In this section, we review the necessary Hodge-theoretic background, defines the key invariants such as genus and combinatorial rank, and apply these concepts to show that Vandermonde cells form genus-zero pairs, ensuring that their canonical forms exist in the Brown–Dúpont sense.

\smallskip
For background on homology and cohomology, we refer to standard references such as \cite{hatcher2002algebraic}. Unless stated otherwise, all (co)homology groups are taken with coefficients in $\mathbb{Q}$. Throughout this section, $Y$ denotes a subvariety of the irreducible complex variety $X$, containing the singular locus of $X$.

In \cite{brown2025positive}, Brown and Dúpont introduced a non-recursive formulation of positive geometries and canonical forms using mixed Hodge theory. This framework endows the cohomology of a complex variety with two compatible filtrations whose interaction controls the existence and uniqueness of canonical forms. Uniqueness, in particular, occurs only when all Hodge numbers $h^{p,0}(H^n)$ with $p>0$ vanish.

\subsection{Hodge-theoretic setup}

Before defining canonical forms in this general setting, we recall the classical notions of pure and mixed Hodge structures, originally due to Hodge and Deligne. 

\begin{definition}[Pure Hodge structure]
A \emph{pure Hodge structure} of weight $w$ on a finite-dimensional $\mathbb{Q}$-vector space $H$ is a decomposition
\[
H_{\mathbb{C}} = \bigoplus_{p+q=w} H^{p,q}
\]
of its complexification satisfying $H^{p,q} = \overline{H^{q,p}}$ for all $p,q$.
\end{definition}

The cohomology of a smooth complex variety naturally carries a pure Hodge structure, where $H^{p,q}$ consists of smooth differential forms of type $(p,q)$.  
Equivalently, one can describe a pure Hodge structure by a decreasing \emph{Hodge filtration} $F$ on $H_{\mathbb{C}}$, defined by
\[
F^i(H_{\mathbb{C}}) := \bigoplus_{p \ge i} H^{p,q}.
\]
This filtration determines the decomposition via
\[
H^{p,q} = F^p(H_{\mathbb{C}}) \cap \overline{F^{q+1}(H_{\mathbb{C}})}.
\]
For singular varieties, the cohomology no longer carries a pure Hodge structure, but it admits an increasing \emph{weight filtration} whose graded pieces are pure Hodge structures.

\begin{definition}[Mixed Hodge structure]
A \emph{mixed Hodge structure} on a finite-dimensional $\mathbb{Q}$-vector space $H$ consists of
\begin{itemize}
    \item an increasing \emph{weight filtration} $W_\bullet$ on $H$, and
    \item a decreasing \emph{Hodge filtration} $F^\bullet$ on its complexification $H_{\mathbb{C}}$,
\end{itemize}
such that, for each weight $w$, the induced filtration $F$ defines a pure Hodge structure of weight $w$ on the graded piece
\[
\operatorname{gr}^W_w H := W_w H / W_{w-1} H.
\]
\end{definition}

If $H$ carries a mixed Hodge structure, we denote by $H^{p,q}$ the $(p,q)$-component of the pure Hodge structure on $\operatorname{gr}^W_{p+q} H$, and by $h^{p,q} = \dim H^{p,q}$ its Hodge numbers.  
A fundamental theorem of Deligne states that the cohomology of every complex algebraic variety is naturally endowed with a canonical mixed Hodge structure.

The Hodge decomposition of cohomology extends to homology via graded Poincaré duality. For $H^n(X,Y)$, the only non-vanishing Hodge numbers are those $h^{p,q}$ with $0 \le p,q \le n$. Dually, the Hodge numbers of $H_n(X,Y)$ are nonzero only for $-n \le p,q \le 0$. Among these, the components $h^{p,0}$ with $p \le 0$ are of particular interest for positive geometries, as they correspond to the dimensions of the spaces of holomorphic $p$-forms.

\subsection{Genus, combinatorial rank, and the iterated residue map}

We first review the definitions of the \emph{genus}, the \emph{combinatorial rank}, and the definition of the canoncial form map ${\rm can} $ of a pair $(X,Y)$ as introduced in \cite{brown2025positive}. For a positive geometry $(X, X_{\geq 0})$, the canonical form $\omega_X$ is holomorphic on $X \setminus Y$. As an $n$-form on an $n$-dimensional variety, $\omega_X$ is closed and therefore defines a class in de~Rham cohomology. This yields a natural map
\[
\Omega^n_{\log}(X \setminus Y) \;\longrightarrow\; H^n(X \setminus Y, \mathbb{C}),
\]
whose image lies in the top step of the Hodge filtration, $F^n\bigl(H^n(X \setminus Y, \mathbb{C})\bigr)$.

In this setting, the construction of canonical forms is more subtle. They are associated to relative homology classes $\sigma \in H_n(X,Y)$ via a morphism
\[
{\rm can}: H_n(X,Y) \;\longrightarrow\; \Omega^n_{\log}(X \setminus Y),
\]
which we define below after introducing the notions of \emph{genus} and \emph{combinatorial rank} for the pair $(X,Y)$.

\begin{definition}
The \emph{genus} of a pair $(X,Y)$ is
$g(X,Y) := \sum_{p>0} h^{p,0}\bigl(H^n(X,Y)\bigr)$,
and its \emph{combinatorial rank} is
$cr(X,Y) := h^{0,0}\bigl(H^n(X,Y)\bigr)$.
\end{definition}
For the mixed Hodge structure on $H_n(X,Y)$ we have
$F^0\bigl(H_n(X,Y), \mathbb{C}\bigr) = \bigoplus_{p \ge 0} H^{p,0}$.
Hence there is a natural projection
\[
F^0\bigl(H_n(X,Y), \mathbb{C}\bigr)
\;\longrightarrow\;
H^{0,0}\bigl(H_n(X,Y), \mathbb{C}\bigr)
= \operatorname{gr}^W_0\bigl(H_n(X,Y), \mathbb{C}\bigr)
= W_0\bigl(H_n(X,Y), \mathbb{C}\bigr) / W_{-1}\bigl(H_n(X,Y), \mathbb{C}\bigr),
\]
whose kernel has dimension $g(X,Y)$ and image has dimension $cr(X,Y)$.  

\medskip
We now define the \emph{iterated residue map} as a composition of natural morphisms in Hodge theory.

\begin{definition}[Iterated residue map]
The \emph{iterated residue map} is the composition
\[
\Omega^n_{\log}(X \setminus Y)
\;\cong\;
F^n\bigl(H^n(X \setminus Y), \mathbb{C}\bigr)
\xrightarrow{\ \mathrm{PD}\ }
F^0\bigl(H_n(X \setminus Y), \mathbb{C}\bigr)
\twoheadrightarrow
\operatorname{gr}^W_0\bigl(H_n(X,Y), \mathbb{C}\bigr),
\]
where the first isomorphism identifies logarithmic forms with the top step of the Hodge filtration, the second arrow is Poincaré duality, and the third is the natural projection to the weight-zero quotient.  
We denote this map by
\[
{\rm R}: \Omega^n_{\log}(X \setminus Y)
\;\longrightarrow\;
\operatorname{gr}^W_0\bigl(H_n(X,Y), \mathbb{C}\bigr).
\]
\end{definition}

The kernel of ${\rm R}$ has dimension $g(X,Y)$. Hence, the iterated residue map can be inverted only when the genus vanishes, i.e.\ when $g(X,Y)=0$.

\begin{definition}[Canonical form map]
For a genus-zero pair $(X,Y)$, the \emph{canonical form map} is defined by composing the natural projection onto the weight-zero quotient
\[
H_n(X,Y) \;\longrightarrow\; \operatorname{gr}^W_0\bigl(H_n(X,Y)_{\CC}\bigr )
\]
with the inverse of the iterated residue map,
\[
{\rm R}^{-1}: \operatorname{gr}^W_0\bigl(H_n(X,Y)_ \mathbb{C}\bigr)
\;\longrightarrow\;
\Omega^n_{\log}(X \setminus Y).
\]
The resulting morphism
\[
{\rm can}: H_n(X,Y) \;\longrightarrow\; \Omega^n_{\log}(X \setminus Y)
\]
assigns to each relative homology class $\sigma$ the logarithmic form ${\rm can}(\sigma)$.
\end{definition}

\subsection{Canonical form map and its properties}

By a \emph{modification} of a pair $(X,Y)$ of complex varieties we mean a morphism of pairs $f:(\tilde{X},\tilde{Y}) \to (X,Y)$ such that $f:\tilde{X}\to X$ is proper, $\tilde{Y}=f^{-1}(Y)$, and the restriction $f|_{\tilde{X}\setminus\tilde{Y}}:\tilde{X}\setminus\tilde{Y}\to X\setminus Y$ is an isomorphism. Standard examples of modifications are blowups, in particular resolutions of singularities.  

With this definition, properties such as additivity under subdivisions, invariance under modification, and functoriality of pullbacks and pushforwards become immediate consequences. By contrast, these properties are significantly more subtle for canonical forms in the sense of Definition~\ref{DefPositiveGeometry} and in \cite{arkani2017positive}.  

\begin{theorem}[\cite{brown2025positive}] 
For a genus-zero pair $(X,Y)$, the canonical form map ${\rm can}$ is linear, invariant under modifications, functorial with respect to pullbacks and pushforwards, and multiplicative under products of pairs. Moreover, taking the relative boundary of a homology class $\sigma$ corresponds to taking the residue of the logarithmic form ${\rm can}(\sigma)$. 
\end{theorem}

This theorem highlights the advantages of the Brown–Dúpont framework for positive geometries. In particular, linearity and functoriality hold formally, whereas in the original formulation these are highly nontrivial results.

We also note that the property that the residue of ${\rm can}(\sigma)$ along a boundary component $D$ is the canonical form of the relative boundary $\partial\sigma $ only makes sense if the boundary $D$ together with its induced boundary is a genus-zero pair as well. In these cases, canonical forms can be determined recursively, similar to the original formulation by Arkani-Hamed, Bai, and Lam as presented in Definition~\ref{DefPositiveGeometry}. There are however genus-zero pairs whose induced boundary pairs have positive genus, showing that the Brown–Dúpont set-up is more general, see \cite[\S2.4.3, Example~5.9]{brown2025positive}. 

An extension of the original definition loosening the recursive requirements on boundaries was foreseeable considering the example of the circle geometry discussed in \cite[\S8, Eq.~(10.8)]{arkani2017positive}. 

\subsection{Application to Vandermonde cells}

We now turn to the pair $(\PP^n,\Pi_{n,d})$. In our set-up, unirational varieties do not admit holomorphic forms. Note however that a priori, holomorphic forms on the smooth locus $X_{reg}$ of a unirational variety $X$ might exist. One consequence of the following theorem is that their pullbacks never extend along the exceptional locus to a globally defined holomorphic form. In the case where (the normalization of) $X$ has rational singularities, no such forms can exist on the smooth locus $X_{reg}$.  
The following theorem shows that even more is true: boundary hypersurfaces of Vandermonde cells have vanishing genus.

Note that having genus zero is generally not a recursive condition (\cite{brown2025positive} 2.4.3). By \cite[Example 3.6]{brown2025positive} this does not immediately follow from having a rational smooth model. The crucial feature of Vandermonde cells is the recursive nature of their boundary stratification where the singular loci of boundary hypersurfaces contain all lower-dimensional boundaries. In particular, all lower-dimensional boundaries $D$ are also unirational and thus have vanishing Hodge number ${\rm dim}(H^{{\rm dim}D}(D))^{{{\rm dim}D},0}$.

\begin{theorem}
\label{Thm:Vanishing geometric genus}
Every irreducible boundary hypersurface of a Vandermonde cell $\Pi_{n,d}$ has vanishing genus. 
\end{theorem}

\begin{proof}
    Let $D$ be an irreducible boundary hypersurface of the Vandermonde cell $\Pi_{n,d}$.  
By Proposition~\ref{prop:Unirational}, the variety $D$ is unirational.  
Hence there exists a rational map 
$\mu \colon \PP^n \dashrightarrow D$,
for instance, the one constructed in the proof of Proposition~\ref{prop:Unirational}.  
We resolve the indeterminacies of $\mu$ by performing a finite sequence of blow-ups along smooth centers $Z_1,\dots,Z_\ell$, obtaining morphisms $p_1$ and $p_2$ such that the following diagram commutes:
    \[
\begin{tikzcd}
 & \arrow[ld, "p_1"', description]  Bl_{Z_1,\dots,Z_l}(\PP^n) \arrow{dr}{p_2} \\
\PP^n \arrow[rr,dashed,"h"] && D
\end{tikzcd}
\]
If the blow-up $Bl_{Z_1,\dots,Z_\ell}(\PP^n)$ is not smooth, we continue blowing up its singular loci until all singularities are resolved, obtaining a resolution $\tilde{D}:=\tilde{Bl_{Z_1,\dots,Z_\ell}(\PP^n)} \to Bl_{Z_1,\dots,Z_\ell}(\PP^n)$. In particular, the composition with $p_2$ resolves the singularities of $D$ (not necessarily minimally, as we may choose to blow up further along smooth centers). 
 
Now consider the genus $g(D)=g(D,\emptyset)= \sum_{p>0} h^{p,0}(H^n(D,\emptyset))$.
We will now show that $H^n(D)^{p,0}=0$ for all $p>0$.

Consider the induced morphism of mixed Hodge structures
$$p_2^{*}\colon H^k(D) \to H^k(\tilde{D}).$$
By a classical theorem (\cite[Theorem~2.29]{Peters2008}), if $D$ is smooth, the morphism $p_2^{*}$ of pure Hodge structures is injective. In the singular case, by \cite[Theorem~5.41]{Peters2008}, the surjective morphism of compact varieties $p_2$ still induces a morphism $p_2^{*}$ of mixed Hodge structures, which remains injective when restricted to the graded piece of weight~$k$. 
In particular, since the mixed Hodge structure on $H^n(\tilde{D})$ of the smooth projective variety $\tilde{D}$ is pure, the top graded piece $gr^W_{n}(H^n(D))$ of $D$ embeds naturally into $H^n(\tilde{D})$.

As $\tilde{D}$ is smooth and rational,
we have 
$$h^{p,0}(H^n(\tilde{D})) = \dim H^0(\tilde{D},\Omega_{\tilde{D}}^p) = 0,$$ 
where the first equality follows from Dolbeault's theorem, and the second from the fact that smooth rational varieties do not admit nontrivial holomorphic $p$-forms for any $p>0$.  

Consequently for $p=n$, we obtain 
$$H^{n}(D)^{n,0} \subset H^n(\tilde{D})^{n,0} = 0.$$

It remains to verify that the lower weight components $H^n(D)^{p,0}$ in the weight filtration also vanish for all $0 < p < n$.
In order to show this, we consider the Mayer--Vietoris sequence for the discriminant square associated with the resolution $\tilde{D} \to D$ (see Definition~5.37 in \cite{Peters2008}).  

Let $Z$ denote the singular locus of $D$.  
By Ursell's Theorem (Theorem~\ref{UrsellTheorem}), the set $Z$ coincides with the union of the closures of all lower-dimensional boundary components of the Vandermonde cell $\Pi_{n,d}$ that are contained in the boundary $D$.  
Let $E := p_2^{-1}(Z)$ denote the exceptional locus of the resolution $p_2$.  
The Mayer--Vietoris sequence of mixed Hodge structures associated with the discriminant square of the resolution 
$p_2\colon \tilde{D} \to D$ is given by 
\begin{align*}
\cdots & \rightarrow H^{n-1}\bigl(E\bigr) 
\rightarrow H^{n}\bigl(D\bigr)
\rightarrow H^{n}\bigl(\tilde{D}\bigr) \oplus H^n\bigl(Z\bigr) 
\rightarrow \cdots.
\end{align*}
In degree $(n-1,0)$, this becomes
\begin{align*}
\cdots & 
\rightarrow H^{n-1}\bigl(E\bigr)^{n-1,0}
\rightarrow H^{n}\bigl(D\bigr)^{n-1,0}
\rightarrow (H^{n}\bigl(\tilde{D}\bigr) \oplus H^{n}\bigl(Z\bigr))^{n-1,0}
\rightarrow \cdots.
\end{align*}

The rightmost and leftmost groups in this sequence vanish.  
Indeed, the Hodge structure on $H^n(\tilde{D})$ for the smooth variety $\tilde{D}$ is pure of weight~$n$, hence $H^n(\tilde{D})^{n-1,0}=0$.  
Since $Z$ is compact of dimension $n-1$, Theorem~5.39 of \cite{Peters2008} implies that $H^n(Z)^{n-1,0}=0$.  
Moreover, $H^{n-1}(E)^{n-1,0}=0$, since $E$ is a simple normal crossings divisor with rational components.  
Altogether, this shows that $H^n(D)^{p,0}=0$ for all $p>0$, as desired.
\end{proof}
We now turn to the study of the genus of the pair $(\PP^n,\Pi_{n,d})$.  
Recall that the genus measures the dimensions of the holomorphic components $H^{p,0}$ of the mixed Hodge structure on $H^n(\PP^{d-1},\Pi_{n,d})$.
In particular, a vanishing genus implies that the canonical form map ${\rm can}$ is well defined and unique.  
The following corollary combines these observations with the vanishing of the geometric genus for boundary hypersurfaces established in Theorem~\ref{Thm:Vanishing geometric genus}.

\begin{corollary}\label{cor:genus-zero}
The pair $(\PP^{d-1},\Pi_{n,d})$ is a genus-zero pair.  
Consequently, the canonical form map ${\rm can}$ is well defined in the sense of Brown–Dúpont. 
\end{corollary}

\begin{proof}
Consider the boundary stratification $(Z_i)_{i=0,\dots,d-2}\subset \Pi_{n,d}$, where $Z_i$ denotes the union of all boundary components of dimension at most~$i$.  
By Theorems~\ref{UrsellTheorem} and~\ref{Thm:TypeofCyclicPolytope}, each stratum satisfies $(Z_i \setminus Z_{i-1}) \neq \emptyset$ for all~$i$.  
By \cite[Proposition~3.10]{brown2025positive}, the genus of the pair $(\PP^{d-1},\mathrm{bd}\,\Pi_{n,d})$ satisfies
\[
g(\PP^{d-1},\mathrm{bd}\,\Pi_{n,d})
\;\leq\;
g(\PP^{d-1},Z_{d-3})
\,+\,
g(\mathrm{bd}\,\Pi_{n,d},Z_{d-3}) .
\]
Since $Z_{d-3}$ has codimension at least~$2$ in $\PP^{d-1}$, \cite[Proposition~3.9]{brown2025positive} implies
\[
g(\PP^{d-1},Z_{d-3}) = g(\PP^{d-1}) = 0.
\]
We now iterate this estimate using the same argument:
\[
g(\mathrm{bd}\,\Pi_{n,d},Z_{d-3})
\;\leq\;
g(\mathrm{bd}\,\Pi_{n,d},Z_{d-4})
\,+\,
g(Z_{d-3},Z_{d-4}).
\]
Denote by $Y_i$ the irreducible components of $\mathrm{bd}\,\Pi_{n,d}$. Because $Z_{d-4}$ has codimension~$2$ in $\mathrm{bd}\,\Pi_{n,d}$, we again obtain, by Theorem~\ref{Thm:Vanishing geometric genus},
\[
g(\mathrm{bd}\,\Pi_{n,d},Z_{d-4})
= g(\mathrm{bd}\,\Pi_{n,d}) \leq \sum g(Y_i) = 0,
\]
where the last inequality follows from \cite[Corollary~3.13 and Proposition~3.9]{brown2025positive}.  
Proceeding inductively in this manner reduces the problem to the lowest-dimensional stratum:
\[
g(Z_1,Z_0)
= g(Z_1),
\]
where the equality follows from \cite[Proposition~3.5]{brown2025positive}.  
Finally, Theorem~\ref{Thm:Vanishing geometric genus} implies $g(Z_1)=0$.  

\smallskip
Hence all intermediate genera vanish, and therefore $g(\PP^{d-1},\mathrm{bd}\,\Pi_{n,d})=0$.  
\end{proof}

\section{The limiting Vandermonde cell \texorpdfstring{$\Pi_d$}{Pi\_d}}\label{sec:limitingcell}

For fixed dimension $d-1$, the Vandermonde cells $\Pi_{n,d}$ form an increasing sequence of bounded semi-algebraic sets whose union converges to a compact limit $\Pi_d$. This limiting object captures the asymptotic geometry of Vandermonde cells as the number of points grows, revealing a transition from algebraic to analytic behavior. In particular, while each finite $\Pi_{n,d}$ is semi-algebraic and admits a well-defined canonical form, the limit $\Pi_d$ develops infinitely many singular boundary components and ceases to be semi-algebraic. Understanding the geometric and analytic nature of this limit and whether a meaningful analogue of the canonical form persists is the main focus of this section.

\medskip In the planar case $d=3$, for example, the upper boundary converges to the cuspidal cubic $\{y^2 = x^3\}$, while the lower boundary acquires infinitely many new components given by the lower parts of the curves $b_k$ from Lemma~\ref{Equations planar boundaries}. From this description it follows that the topological boundary of the limiting Vandermonde cell contains infinitely many singular points and is therefore not semi-algebraic \cite[§2.25]{WonderfulGeometry}. This is already apparent in the planar case, where the points $(\tfrac{1}{k},\tfrac{1}{k^2})$ are singular for all $k \in \mathbb{N}$. Since planar Vandermonde cells arise as projections of higher-dimensional ones, it follows that their higher-dimensional limits are not semi-algebraic either.
Consequently, it is not clear what a \emph{canonical form} for $\lim_{n \to \infty} \Pi_{n,d}$ could or should be.  

Locally, the coefficient function of a rational form $\omega$ on $\PP^n$ can be expressed as $\tfrac{f}{g}$, where $f$ and $g$ are holomorphic. If $\omega$ were required to have poles along all irreducible boundary components, then $g$ would vanish on infinitely many distinct components in any sufficiently small analytic neighborhood of $0$. By the Weierstrass preparation theorem, this is impossible. Hence no rational differential form can have poles precisely along the boundary of $\Pi_d$.  

An alternative route to defining canonical forms for full-dimensional polytopes is through dual volumes; see \cite{gao2024dualmixedvolume} for a detailed discussion of dual mixed volume functions and their relation to positive geometries. In \cite{mazzucchelli2025canonicalformsdualvolumes}, the dual volume approach is refined through additional flexibility in regard to the used measure. 
Recall that for a subset $A \subset \RR^d$, its polar dual is defined as
\[
A^\vee := \{\, y \in \RR^d \;\mid\; \langle y, x \rangle \geq -1 \text{ for all } x \in A \,\}.
\]
We denote the Euclidean volume by ``$\Vol$'', normalized so that the standard simplex has volume~$1$.  

\begin{proposition}[\cite{arkani2017positive}, Equation~7.180]\label{ConvexApproximation}
Let $P \subset \RR^d \subset \PP^d$ be a full-dimensional projective polytope. 
Then its canonical form $\omega_P$ is given by 
\[
\omega_P(x)= \Vol\bigl((P-x)^{\vee}\bigr)\, dx_1 \cdots dx_d.
\]
\end{proposition}

The coefficient $\Vol((P-x)^{\vee})$ is a well-defined holomorphic function for all $x$ in the interior of $P$, and extends analytically to all $x \in \PP^n$ not contained in the span of a facet of~$P$.  

The polar dual of a set $A$ coincides with that of its convex hull, ${\rm conv}(A)$. Hence, the canonical forms of non-convex candidates $A$ cannot be expressed directly via dual volumes. However, if $A$ can be subdivided into convex polytopes, the dual volume formula can still be applied piecewise. For instance, the set $A$ from \cite[Figure~12]{arkani2017positive} can be decomposed into two triangles, a left triangle $T_l$ and a right triangle $T_r$, distinguished by the sign of their horizontal coordinate. In this case,
\[
\omega_A(x) \;=\; \omega_{T_l}(x) + \omega_{T_r}(x) 
\;=\; \Vol\bigl((T_l - x)^{\vee}\bigr) + \Vol\bigl((T_r - x)^{\vee}\bigr).
\]

This viewpoint also sheds light on the analytic continuation of $\Vol((P - x)^{\vee})$ for $x$ outside the interior of $P$. For example, if $x \in T_l$ lies on the adjoint line (noting that for non-convex sets the adjoint may intersect the interior), then
\[
\Vol\bigl((T_r - x)^{\vee}\bigr) = -\,\Vol\bigl((T_l - x)^{\vee}\bigr),
\]
where the right-hand side represents, up to sign, the actual volume of a polytope, while the left-hand side is obtained via analytic continuation.

We summarize this observation in the following proposition.

\begin{proposition}\label{prop:dual}
    The value of the canonical function $\omega_P(x)$ for a point $x \notin P^{\circ}$ 
    can be determined by exclusively computing dual volumes of other polytopes containing $x$.   
\end{proposition}

\begin{proof}
    Consider a larger polytope $\mathcal{P}$ containing $x$ and a subdivision of $\mathcal{P}$ into polytopes $(P_i)_{i = 0, \dots, n}$, intersecting along common faces, including $P := P_0$, such that there exists a polytope $P_{i_0}$ containing $x$ in its interior. 
    The value $\omega_P(x)$ is the sum $\sum_{i=0}^n \omega_{P_i}(x)$. Consider a polytope $P_{i_1}$ that intersects $P_{i_0}$ in a facet. By refining the subdivision, we may assume that $P_{i_0} \cup P_{i_1}$ is convex. Then we have
    \[
    \omega_{P_{i_0} \cup P_{i_1}}(x) = \omega_{P_{i_0}}(x) + \omega_{P_{i_1}}(x),
    \]
    where two of the three terms can be computed as actual volumes. Iterating this procedure, with potential refinement or coarsening of the triangulation, we can compute the value of all canonical functions $\omega_{P_i}(x)$.     
\end{proof}

The following example illustrates the proof of Proposition~\ref{prop:dual}.
\begin{example}
   In order to determine $\omega_{P_3}(x)$, we first compute $\omega_{P_1}(x)$ and $\omega_{P_2}(x)$ by evaluating the volumes of 
$(P_{i_0} - x)^{\vee}$, $((P_{i_0} \cup P_1) - x)^{\vee}$, and $((P_{i_0} \cup P_2) - x)^{\vee}$. 
These results are then used to obtain $\omega_{P_1}(x)$ and $\omega_{P_2}(x)$. 
Finally, $\omega_{P_3}(x)$ is determined via the equality
\[
\omega_{P_{i_0}}(x) + \omega_{P_1}(x) + \omega_{P_2}(x) + \omega_{P_3}(x)
= \omega_{P_{i_0} \cup P_1 \cup P_2 \cup P_3}(x).
\]

\begin{figure}[ht!]
\centering
\includegraphics[width=50mm]{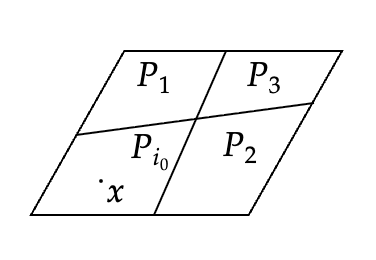}
\caption{The polytopal subdivision used to determine $\omega_{P_3}(x)$.}
\label{fig:subdivision}
\end{figure}

\end{example}

Note that the polar dual $P^{\vee}$ of a polytope $P$ is itself a polytope if and only if $0$ lies in the interior $P^{\circ}$ of $P$. In general, canonical forms are defined only up to the orientation of the smooth part of the interior, or equivalently, up to a choice of sign on each connected component of the interior. In this work, we restrict attention to ambient positive geometries with connected interior, so the orientation amounts to a single global sign choice.

Theorem~\ref{ConvexApproximation} can also be applied to convex semi-algebraic sets defined by non-linear equations by approximating them with an increasing sequence of polytopes.  

\begin{proposition}\label{PolytopeApproximation}
Let $X_{\geq 0} \subset \RR^d \subset \PP^d$ be a convex semi-algebraic set such that the pair $(\PP^d,X_{\geq 0})$ is a positive geometry (in the sense of Definition~\ref{DefPositiveGeometry}) with canonical form $\omega_{X_{\geq 0}}$. Suppose $(P_j)_{j \in \mathbb{N}}$ is a sequence of full-dimensional polytopes satisfying: 
\begin{enumerate}
    \item Each polytope $P_j$ is contained in $X_{\geq 0}$, the interiors of distinct polytopes are disjoint, and whenever $P_i \cap P_j \neq \emptyset$, the intersection is a common face of both.
    \item The sequence $\bigl(\bigcup_{j \leq i} P_j\bigr)_{i \in \mathbb{N}}$ is increasing, convex, and converges to $X_{\geq 0}$, i.e., for every $x \in X_{\geq 0}$ there exists $j \in \mathbb{N}$ with $x \in P_j$.
\end{enumerate}
Then the canonical forms converge in the limit:
\[
\lim_{j \to \infty} \, \omega_{\cup_{i \leq j} P_i} \;=\; \omega_{X_{\geq 0}}.
\]
\end{proposition}

\begin{proof}
Consider the limit
\[
\lim_{j\to \infty} \, \omega\!\left(\bigcup_{i \leq j} P_{i}\right)(x).
\]
For any point $x$ in the interior of some $P_{i_0}$, we have
\[
\lim_{j\to \infty} \omega\!\left(\bigcup_{i \leq j} P_{i}\right)(x)
= \lim_{j\to \infty} \sum_{i \leq j} \omega(P_{i})(x)
= \lim_{j\to \infty} \sum_{i \leq j} \Vol\!\bigl((P_i - x)^{\vee}\bigr).
\]
This can be rewritten as
\[
\lim_{j\to \infty} \left( \Vol\!\bigl((P_{i_0}-x)^{\vee}\bigr) \;+\; \sum_{\substack{i \leq j \\ i \neq i_0}} \Vol\!\bigl((P_i - x)^{\vee}\bigr) \right).
\]
The first summand is a fixed positive number (the orientation is chosen so that the signed volume is positive; see the remark above). The second summand is negative: for every $j \geq i_0$, we have
\[
\Vol\!\left(\bigl(\bigcup_{i \leq j} P_i - x \bigr)^{\vee}\right) 
= \omega\!\left(\bigcup_{i \leq j} P_i\right)(x)
= \Vol\!\bigl((P_{i_0}-x)^{\vee}\bigr) 
+ \sum_{\substack{i \leq j \\ i \neq i_0}} \Vol\!\bigl((P_i - x)^{\vee}\bigr),
\]
and
\[
0 < \Vol\!\left(\bigl(\bigcup_{i \leq j} P_i - x \bigr)^{\vee}\right) \leq \Vol\!\bigl((P_{i_0}-x)^{\vee}\bigr),
\]
since $P_{i_0} \subset \bigcup_{i \leq j} P_i$.  
If $\mathcal{P}$ is a polytope containing $X_{\geq 0}$, then the sequence
\[
\left( \Vol\!\left(\bigl(\bigcup_{i \leq j} P_i - x \bigr)^{\vee}\right) \right)_{j \in \mathbb{N}}
\]
is bounded below and hence convergent. Its limit coincides with the canonical form of $X_{\geq 0}$, since the difference between $\bigcup_{i \leq j} P_i$ and $X_{\geq 0}$ tends to the empty set as $j \to \infty$.
\end{proof}

\noindent{\bf Non-convex semi-algebraic sets.}
Proposition~\ref{PolytopeApproximation} can be extended to certain non-convex sets $X_{\geq 0}$. Consider the convex hull ${\rm conv}(X_{\geq 0})$, and assume that each connected component of ${\rm conv}(X_{\geq 0}) \setminus X_{\geq 0}$ is convex. In this case, the canonical form of $X_{\geq 0}$ can be obtained from the \emph{outer triangulation}
\[
{\rm conv}(X_{\geq 0}) \;\setminus\; \{ \text{components of } {\rm conv}(X_{\geq 0}) \setminus X_{\geq 0} \}.
\]
In particular, each topological boundary of a planar Vandermonde cell $\Pi_{n,3}$ is concave, allowing the canonical forms $\Omega_{\Pi_{n,3}}$ to be reconstructed using this approach. Specifically, every component of the complement ${\rm conv}(\Pi_{n,3}) \setminus \Pi_{n,3}$ is a semi-algebraic set of type $A_{III}$, while ${\rm conv}(\Pi_{n,3})$ itself is a cyclic polytope. Hence one can avoid computing canonical forms for the semi-algebraic sets $A_{I}$ and $A_{II}$ from Proposition~\ref{Prop:SemiAlgTypes}.  

\medskip
\noindent{\bf The limiting Vandermonde cell.}
The outer-triangulation viewpoint also suggests a natural analogue of the canonical form for the limiting Vandermonde cell $\Pi_{3}$:
\[
\Vol\bigl((\Pi_{3}-x)^{\vee}\bigr) \;-\; \sum_{i=1}^{\infty} \omega_{A_{III,i}},
\]
where $A_{III,i}$ denotes the semi-algebraic set of type $A_{III}$ given by ${\rm conv}(\Pi_{i+1,3}) \setminus \Pi_{i+1,3}$. Considering dual volumes shows that this infinite sum converges on a full-dimensional subset of the interior.  

Geometrically, ${\rm conv}(\Pi_{3})$ is an \emph{infinite cyclic polytope}, which may be viewed as a limiting analogue of an amplituhedron. However, this object differs from the limit amplituhedron studied by Köffler and Sinn in \cite{koefler2025taking}. It therefore represents a distinct type of limiting positive geometry, whose analytic properties are inherited from the finite Vandermonde cells but whose boundary is no longer semi-algebraic.

\bibliographystyle{plain}
\bibliography{RefsUpdated.bib}

\bigskip
{\small\noindent {\bf Authors' addresses}
\medskip

\noindent{
Departments of Computer Science and Mathematics, KU Leuven}\hfill {\tt  fatemeh.mohammadi@kuleuven.be } 

\medskip
\noindent{
Department of Mathematics, KU Leuven}\hfill {\tt  sebastian.seemann@kuleuven.be }

\end{document}